\documentclass{amsart}
\usepackage{amssymb, latexsym, amsthm}
\usepackage{url}

\DeclareMathOperator{\Aut}{\mathrm{Aut}}
\DeclareMathOperator{\closure}{\mathrm{closure}}

\DeclareMathOperator{\Iso}{\mathrm{Iso}}

\DeclareMathOperator{\iim}{\mathrm{im}}
\DeclareMathOperator{\Norm}{\mathrm{Norm}}

\DeclareMathOperator{\Orb}{\mathrm{Orb}}

\DeclareMathOperator{\spann}{\mathrm{span}}

\DeclareMathOperator{\Tor}{\mathrm{Tor}}

\begin{document} 
\bibliographystyle{alpha}

\newtheorem{theorem}{Theorem}[section]
\renewcommand{\thetheorem}{\arabic{section}.\arabic{theorem}}
\newtheorem{proposition}[theorem]{Proposition}
\renewcommand{\theproposition}{\arabic{section}.\arabic{theorem}}
\newtheorem{lemma}[theorem]{Lemma}
\renewcommand{\thelemma}{\arabic{section}.\arabic{theorem}}
\newtheorem{corollary}[theorem]{Corollary}
\renewcommand{\thecorollary}{\arabic{section}.\arabic{theorem}}
\newtheorem{conjecture}[theorem]{Conjecture}
\renewcommand{\theconjecture}{\arabic{section}.\arabic{theorem}}
\newtheorem{definition}[theorem]{Definition}
\renewcommand{\thedefinition}{\arabic{section}.\arabic{theorem}}
\newtheorem{question}[theorem]{Question}
\renewcommand{\thequestion}{\arabic{section}.\arabic{theorem}}
\newtheorem*{claim*}{Claim}
\newtheorem{claim}[theorem]{Claim}
\renewcommand{\theclaim}{\arabic{section}.\arabic{theorem}}
\newcommand{\mc}{\mathcal}
\newcommand{\A}{\mc{A}}
\newcommand{\B}{\mc{B}}
\newcommand{\cc}{\mc{C}}
\newcommand{\D}{\mc{D}}
\newcommand{\E}{\mc{E}}
\newcommand{\F}{\mc{F}}
\newcommand{\G}{\mc{G}}
\newcommand{\hH}{\mc{H}}
\newcommand{\FN}{\F_n}
\newcommand{\I}{\mc{I}}
\newcommand{\J}{\mc{J}}
\newcommand{\K}{\mc{K}}
\newcommand{\eL}{\mc{L}}
\newcommand{\M}{\mc{M}}
\newcommand{\eN}{\mc{N}}
\newcommand{\qq}{\mc{Q}} 
\newcommand{\sS}{\mc{S}}
\newcommand{\U}{\mc{U}}
\newcommand{\V}{\mc{V}}
\newcommand{\W}{\mc{W}}
\newcommand{\X}{\mc{X}}
\newcommand{\Y}{\mc{Y}}
\newcommand{\zZ}{\mc{Z}}
\newcommand{\wS}{\widetilde{S}}
\newcommand{\wT}{\widetilde{T}}
\newcommand{\C}{\mathbb{C}}
\newcommand{\R}{\mathbb{R}}
\newcommand{\N}{\mathbb{N}}
\newcommand{\Q}{\mathbb{Q}}
\newcommand{\Z}{\mathbb{Z}}
\newcommand{\aA}{\mathfrak A}
\newcommand{\bB}{\mathfrak B}
\newcommand{\ff}{\mathfrak F}
\newcommand{\fp}{\mathfrak p}
\newcommand{\fb}{f_{\beta}}
\newcommand{\fg}{f_{\gamma}}
\newcommand{\gb}{g_{\beta}}
\newcommand{\ep}{\varepsilon}
\newcommand{\vphi}{\varphi}
\newcommand{\bo}{\boldsymbol 0}
\newcommand{\ba}{\boldsymbol a}
\newcommand{\bb}{\boldsymbol b}
\newcommand{\bm}{\boldsymbol m}
\newcommand{\bgamma}{\boldsymbol \gamma}
\newcommand{\bt}{\boldsymbol t}
\newcommand{\bu}{\boldsymbol u}
\newcommand{\bv}{\boldsymbol v}
\newcommand{\bx}{\boldsymbol x}
\newcommand{\bwy}{\boldsymbol y}
\newcommand{\bxi}{\boldsymbol \xi}
\newcommand{\bbeta}{\boldsymbol \eta} 
\newcommand{\bw}{\boldsymbol w}
\newcommand{\bz}{\boldsymbol z}
\newcommand{\whG}{\widehat{G}}
\newcommand{\oK}{\overline{K}}
\newcommand{\oKt}{\overline{K}^{\times}}
\newcommand{\oQ}{\overline{\Q}}
\newcommand{\oq}{\oQ^{\times}}
\newcommand{\oQt}{\oQ^{\times}/\Tor\bigl(\oQ^{\times}\bigr)}
\newcommand{\ot}{\Tor\bigl(\oQ^{\times}\bigr)}
\newcommand{\h}{\frac12}
\newcommand{\hh}{\tfrac12}
\newcommand{\dt}{\text{\rm d}t}
\newcommand{\dx}{\text{\rm d}x}
\newcommand{\dy}{\text{\rm d}y}
\newcommand{\dmu}{\text{\rm d}\mu}
\newcommand{\dnu}{\text{\rm d}\nu}
\newcommand{\dla}{\text{\rm d}\lambda}
\newcommand{\dlav}{\text{\rm d}\lambda_v}
\newcommand{\trho}{\widetilde{\rho}}
\newcommand{\dtrho}{\text{\rm d}\widetilde{\rho}}
\newcommand{\drho}{\text{\rm d}\rho}
\newcommand{\NN}{\mathbb{N}}
\newcommand{\GG}{\mathbb{G}}
\newcommand{\QQ}{\mathbb{Q}}
\newcommand{\QQbar}{\overline{\QQ}}
\newcommand{\Kdiv}{K^{\text{div}}} 
\newcommand{\di}{\textup{div}}
\newcommand{\hg}{h_\Gamma}
\newcommand{\hk}{h_k} 
\newcommand{\tors}{\textup{tors}}
\newcommand{\vep}{\varepsilon}
\newcommand{\mcF}{\mathcal{F}}
\newcommand{\Gal}{\operatorname{Gal}}
\def\today{\number\time, \ifcase\month\or
January\or February\or March\or April\or May\or June\or
July\or August\or September\or October\or November\or December\fi
\space\number\day, \number\year}
 
\title[Weil height]{Multiplicative approximation\\by the Weil height\hfill}
\author{Robert Grizzard and Jeffrey~D.~Vaaler}
\subjclass[2010]{11G50, 11J25, 11R04, 46B04}
\keywords{Weil height, projection operators}
\thanks{Research of the second author was supported by a grant from the National Security Agency, H92380-12-1-0254.}
\address{Department of Mathematics, Lafayette College, Easton, Pennsylvania 18042 USA}
\email{grizzarr@lafayette.edu}
\address{Department of Mathematics, University of Texas, Austin, Texas 78712 USA}
\email{vaaler@math.utexas.edu}
\numberwithin{equation}{section}

\begin{abstract}   Let $K/\Q$ be an algebraic extension of fields, and let $\alpha \not= 0$ be contained
in an algebraic closure of $K$.  If $\alpha$ can be approximated by roots of numbers in $K^{\times}$ with respect to the 
Weil height, we prove that some nonzero integer power of $\alpha$ must belong to $K^{\times}$.  More generally, let 
$K_1, K_2, \dots , K_N$, be algebraic extensions of $\Q$ such that each pair of extensions includes one which is a (possibly infinite) 
Galois extension of a common subfield.  If $\alpha \not= 0$ can be approximated by a product of roots of numbers from each 
$K_n$ with respect to the Weil height, we prove that some nonzero integer power of $\alpha$ must belong to the 
multiplicative group $K_1^{\times} K_2^{\times} \cdots K_N^{\times}$.  Our proof of the more general result uses 
methods from functional analysis.
\end{abstract}

\maketitle

\section{Introduction}

Throughout this paper we work with algebraic extensions $K/\Q$ contained in a common algebraic closure $\QQbar$.  We
write $K^{\times}$ for the multiplicative group of $K$, $\Tor\bigl(K^{\times}\bigr)$ for its torsion subgroup, and
\begin{equation*}\label{height1}
h : \oq \rightarrow [0, \infty)
\end{equation*}
for the absolute, logarithmic Weil height.  
If $\Gamma$ is a subgroup of $\QQbar^\times,$ a height function relative to $\Gamma$, written
\begin{equation*}\label{height2}
h_\Gamma :  \QQbar^\times \to [0,\infty),
\end{equation*}
is defined by
\begin{equation*}\label{hgamma1}
h_\Gamma(\alpha) = \inf_{\gamma \in \Gamma} h(\alpha/\gamma).
\end{equation*}
Such height functions have been considered previously in \cite{dlmf2008} and \cite{FM2012}.  If we interpret $h(\alpha/\beta)$ as 
a semi-distance between $\alpha$ and $\beta$, then $h_\Gamma(\alpha)$ is a semi-distance from $\alpha$ to the subgroup 
$\Gamma$.  Of course, taking $\Gamma$ to be a subgroup of $\ot$, reproduces the original height function $h$.

For an algebraic extension $K/\Q$, we define
\begin{equation*}\label{hgamma4}
\Kdiv = \{\gamma \in \QQbar^\times : \text{$\gamma^m \in K^{\times}$ for some $m \not= 0$ in $\Z$}\}.
\end{equation*}
It follows that $\Kdiv$ is a divisible subgroup of $\oq$ that contains $K^\times$.  
For $\Gamma = K^{\di}$ and $\alpha$ in $\oq$, we find that
\begin{equation}\label{hgamma8}
\begin{split}
h_{\Gamma}(\alpha) 
	&= \inf \big\{h(\alpha/\gamma) : \text{$\gamma \in \oq$, and $\gamma^m \in K^{\times}$ for some $m \not= 0$ in $\Z$}\big\}\\
	&= \inf \big\{|m|^{-1} h(\alpha^m/\beta) : \text{$\beta \in K^{\times}$ and $m \not= 0$ in $\Z$}\big\}.
\end{split}
\end{equation}
Using (\ref{hgamma8}) we define a map
\begin{equation*}\label{hgamma12}
V_K : \oq \rightarrow [0, \infty),
\end{equation*}
by setting
\begin{equation*}\label{hgamma16}
V_K(\alpha) = h_{\Gamma}(\alpha).
\end{equation*}

For an algebraic extension $K/\Q$, we write $\Aut(\QQbar/K)$ for the profinite group of automorphisms of $\oQ$ that fix 
each element of $K$.  Then we define a second map
\begin{equation*}\label{hgamma20}
W_K : \oq \rightarrow [0, \infty)
\end{equation*}
by
\begin{equation}\label{widthk}
W_K(\alpha) = \sup \big\{h\left(\sigma \alpha / \alpha\right) : \text{$\sigma \in \Aut(\oQ/K)$} \big\}.
\end{equation}
As $\alpha$ in $\oq$ has only finitely many distinct conjugates over $K$, it is clear that the supremum in (\ref{widthk})
could be replaced by a maximum.
Notice that $W_K(\alpha) = 0$ if and only if all conjugates of $\alpha$ over $K$ are multiples of $\alpha$ by a root of unity, 
which is equivalent to saying that $\alpha$ lies in $K^\di$.  Indeed, if all conjugates of 
$\alpha$ over $K$ are multiples of $\alpha$ by roots of unity, it follows that some power of $\Norm_{K(\alpha)/K}(\alpha)$, 
that is, some power of the product of the conjugates of $\alpha$ over $K$, is equal to a power of $\alpha$.  

We begin with the following result, which is simple enough that we include its proof immediately.

\begin{theorem}\label{thma}
Let $K$ be a subfield of $\QQbar$, and let $\alpha \in \QQbar^\times$.  Then
\begin{equation}\label{widthk3}
\hh W_K(\alpha) \leq V_K(\alpha) \leq W_K(\alpha).
\end{equation}
\end{theorem}

\begin{proof}
Observe that, for each $\alpha \in \QQbar^\times$ we have
\begin{equation}\label{hkinf}
V_K(\alpha) = \inf_{\gamma \in K^\di} h(\alpha/\gamma) = \inf_{\substack{\beta \in K^\times \\ 0 \not= m \in \Z}} h(\alpha/\beta^{1/m}) 
		    = \inf_{\substack{\beta \in K^\times \\ 0 \not= m \in \Z}} |m|^{-1} h(\alpha^m/\beta).
\end{equation}
Let $\tau$ be an element of $\Aut(\oQ/K)$ such that $W_{K}(\alpha) = h(\tau \alpha/\alpha)$.  Then for each 
$0 \not= m \in \Z$ and $\beta \in K$ we have
\begin{align}\label{hkinf3}
\begin{split}
W_{K}(\alpha) &= |m|^{-1} h\left(\tau(\alpha^m)/\alpha^m\right) 
		= |m|^{-1} h\left(\tau\left(\alpha^m/\beta\right) \bigl(\beta/\alpha^m\bigr)\right) \\[1em]
                        &\leq |m|^{-1} h\left(\tau\left(\alpha^m/\beta\right)\right) + |m|^{-1} h\left(\beta/\alpha^m\right)   
                 = 2 |m|^{-1} h\left(\alpha^m/\beta\right).
\end{split}                 
\end{align}
Taking the infimum on the right of (\ref{hkinf3}), as in (\ref{hkinf}), we obtain the inequality on the left of (\ref{widthk3}).

Let $\alpha = \alpha_1,\dots,\alpha_n$ denote the Galois conjugates of $\alpha$ over $K$.  Then 
$\eta = \alpha_1\cdots\alpha_n$ lies in $K^\times$, and therefore
\begin{align*}
V_K(\alpha) \leq n^{-1} h(\alpha^n/\eta) \leq n^{-1} \sum_{i=1}^n h(\alpha / \alpha_i) \leq \max_{1\leq i \leq n} h(\alpha/\alpha_i) 
		    = W_K(\alpha).
\end{align*} 
This verifies the inequality on the right of (\ref{widthk3}).
\end{proof}

Since $W_K(\alpha) = 0$ if and only if $\alpha \in K^\di$, Theorem \ref{thma} implies that $V_K(\alpha) = 0$ if and only if 
$\alpha \in K^\di$.  This result was obtained in \cite[Theorem 2]{dlmf2008} when $K$ is a number field.  Alternatively,
if $\alpha$ in $\oq$ can be approximated with respect to the Weil height by elements of $K^{\di}$, then $\alpha$ must
belong to $K^{\di}$.  We state this more precisely in the following result.

\begin{corollary}\label{corintro2}
Let $K$ be a subfield of $\QQbar$, and let $\alpha \in \QQbar^\times$.  Assume that for every $\vep>0$ there is 
an integer $m \not = 0$ and an element $\beta \in K^\times$ such that
\begin{equation}\label{hkinf5}
h(\alpha^m/\beta) < \vep |m|.
\end{equation} 
Then there is an integer $n \not = 0$ such that $\alpha^n$ belongs to $K^{\times}$.
\end{corollary}

\begin{proof}  For $\vep > 0$, let $m \not= 0$ in $\Z$, and $\beta$ in $K^{\times}$, satisfy (\ref{hkinf5}).  Select $\gamma$
in $K^{\di}$ so that $\gamma^m = \beta$.  It follows that
\begin{equation*}\label{hkinf8}
h(\alpha/\gamma) = |m|^{-1} h(\alpha^m/\gamma^m) = |m|^{-1} h(\alpha^m/\beta) < \vep.
\end{equation*}
As $\vep > 0$ is arbitrary, we conclude that $V_K(\alpha) = 0$.  Hence $\alpha$ belongs to $K^{\di}$, and this is the 
assertion to be proved.
\end{proof}

Our main result is a generalization of Corollary \ref{corintro2} in which the field $K$ is replaced by a finite collection 
of fields $K_1, K_2, \dots , K_N$.  If $L = K_1 K_2 \cdots K_N$ is the composite field, then the product of multiplicative groups
\begin{equation*}\label{hkinf6}
K_1^{\times} K_2^{\times} \cdots K_N^{\times} \subseteq L^{\times},
\end{equation*}
and in general the subgroup $K_1^{\times} K_2^{\times} \cdots K_N^{\times}$ can have infinite index in $L^{\times}$.
For our purposes we require the divisible group
\begin{equation*}\label{hkinf7}
(K_1^{\times} K_2^{\times} \cdots K_N^{\times})^\di 
	= \{\gamma \in \oq : \text{$\gamma^m \in K_1^{\times} K_2^{\times} \cdots K_N^{\times} $ for some $m \not= 0$ in $\Z$}\}.
\end{equation*}
When working with several fields we assume that for $n_1 \not= n_2$, at least 
one of the two fields $K_{n_1}$ or $K_{n_2}$ is a (possibly infinite) Galois extension of their common 
subfield $K_{n_1} \cap K_{n_2}$.  

\begin{theorem}\label{thmintro3}  Let $K_1, K_2, \dots , K_N$, be a collection of fields such that
\begin{equation}\label{intro2}
\Q \subseteq K_n \subseteq \oQ,\quad\text{for each $n = 1, 2, \dots , N$},
\end{equation}
and for each pair of integers $n_1 \not= n_2$, either
\begin{equation}\label{intro3}
K_{n_1}/(K_{n_1} \cap K_{n_2}) \quad\text{is a (possibly infinite) Galois extension,}
\end{equation}
or
\begin{equation}\label{intro4}
K_{n_2}/(K_{n_1} \cap K_{n_2}) \quad\text{is a (possibly infinite) Galois extension.}
\end{equation}
Let $\alpha$ be an element of $\oq$, and let 
\begin{equation*}\label{intro5}
\Gamma = (K_1^\times K_2^\times \cdots K_N^\times)^\di \subseteq \QQbar^\times.
\end{equation*}
Then $h_\Gamma(\alpha) = 0$ if and only if $\alpha \in \Gamma.$
\end{theorem}

The following result is an alternative statement of Theorem \ref{thmintro3} that generalizes Corollary \ref{corintro2}.

\begin{corollary}\label{corintro4}  Let $K_1, K_2, \dots , K_N$, be a collection of fields that satisfy {\rm (\ref{intro2})},
{\rm (\ref{intro3})}, and {\rm (\ref{intro4})}.   Assume that for every $\ep > 0$, there exists an integer $m \not= 0$, and 
points $\beta_n$ in $K_n^{\times}$, for $n = 1, 2, \dots , N$, such that
\begin{equation*}\label{intro13}
h\bigl(\alpha^m \beta_1^{-1}\beta_2^{-1} \cdots \beta_N^{-1}\bigr) < \ep |m|.
\end{equation*} 
Then there exists an integer $q \not= 0$ such that $\alpha^q \in K_1^{\times} K_2^{\times} \cdots K_N^{\times}$. 
\end{corollary} 

In order to prove Theorem \ref{thmintro3}, we exploit the structure of the quotient group
\begin{equation}\label{intro17}
\G = \oQt
\end{equation}
as a $\Q$-vector space with a (vector space) norm given by the Weil height.  We embed $\G$ isometrically in the  Banach space 
$\X$ considered in \cite{all2009}, and defined below in (\ref{ht20}).  Working in $\X$ it is natural to employ methods from functional 
analysis.  Using the hypotheses in Theorem \ref{thmintro3} we will prove that the closure of the image of the subgroup 
$\Gamma = (K_1^\times K_2^\times \cdots K_N^\times)^\di$ in $\X$ is a complemented subspace of $\X$.  The corresponding
continuous linear projection is defined using a generalization of the classical field norm map.  The
details are summarized in the next section, and at the end of that section we summarize the proof of our main result.

\section{The Banach space determined by the Weil height}\label{avb}

If $\alpha$ belongs to $\oq$, and $\zeta$ is an element of the torsion subgroup $\ot$, then it is well known that
$h(\alpha) = h(\zeta \alpha)$.  It follows that the height $h$ is constant on cosets of the quotient group (\ref{intro17}).
Therefore $h$ is well defined as a map
\begin{equation*}\label{intro25}
h : \G \rightarrow [0, \infty),
\end{equation*}
and elementary properties of the height imply that $(\alpha, \beta) \mapsto h\bigl(\alpha \beta^{-1}\bigr)$ defines a metric
on $\G$.  Moreover, if $r/s$ is a rational number, where $r$ and $s$ are relatively prime integers, $s$ is positive, and 
$\alpha$ belongs to $\G$, then
\begin{equation}\label{intro30}
(r/s, \alpha) \mapsto \alpha^{r/s}
\end{equation}
is a well defined scalar multiplication.  As discussed in \cite[section 1]{all2009}, or in \cite[section 1]{vaaler2012}, the 
group $\G$ is a vector space over the field $\Q$ of rational numbers, written multiplicatively, and with scalar 
multiplication defined by (\ref{intro30}).  The identity (see \cite[Lemma 1.5.18]{bombieri2006})
\begin{equation*}\label{intro35}
h\bigl(\alpha^{r/s}\bigr) = |r/s| h(\alpha),
\end{equation*}
implies that $\alpha \mapsto h(\alpha)$ is a norm on the $\Q$-vector space $\G$.  It follows that the height induces a norm
topology in $\G$.  Working in the quotient group $\G$, Corollary \ref{corintro2} and Theorem \ref{thmintro3} both assert that 
certain subsets of $\G$ are closed in the norm topology of $\G$ induced by the Weil height.

Let $K/\Q$ be an algebraic extension of fields, and let
\begin{equation*}\label{intro45}
\vphi : \oq \rightarrow \G = \oQt
\end{equation*}
be the canonical homomorphism.  We write
\begin{equation*}\label{intro50}
\G_K = \big\{\vphi(\alpha) : \alpha \in K^{\times}\big\}
\end{equation*}
for the image of $K^{\times}$ in $\G$, so that $\G_K$ is a subgroup of $\G$, and $\G_K$ is isomorphic to 
$K^{\times}/\Tor\bigl(K^{\times}\bigr)$.  Let $\spann_{\Q} \G_K$ denote the $\Q$-linear subspace of $\G$ generated
(multiplicatively) by the elements of $\G_K$.  It follows easily that
\begin{equation}\label{intro55}
\spann_{\Q} \G_K = \big\{\beta \in \G : \text{there exists $m \not= 0$ in $\Z$ such that $\beta^m \in \G_K$}\big\}.
\end{equation}
Let $\alpha$ be an element of $\G$.  Corollary \ref{corintro2} asserts that if $\alpha$ is a limit point of 
$\spann_{\Q} \G_K$, then $\alpha$ is an element of $\spann_{\Q} \G_K$.  Thus we get the following alternative
statement of Corollary \ref{corintro2}.

\begin{corollary}\label{corintro3}  Let $K/\Q$ be an algebraic extension of fields, and let $\spann_{\Q} \G_K$
be the $\Q$-linear subspace of $\G$ generated by $\G_K$.  Then $\spann_{\Q} \G_K$ is closed in the norm 
topology of $\G$ induced by the height.
\end{corollary}

It is possible to formulate an alternative statement of Theorem \ref{thmintro3} which also asserts that a certain
$\Q$-linear subspace of $\G$ is closed.  We prefer, however, to introduce an isomorphic copy of $\G$
that is a dense subset of the real Banach space $\X$ defined in (\ref{ht20}).  In \cite[Theorem 1]{all2009} it was shown
that the space $\X$ is isomorphic to the completion of $\G$ with respect to the norm induced 
by the height.  The methods we use here from functional analysis are more naturally employed in the Banach space $\X$.

Let $k$ be an algebraic number field of degree $d$ over $\Q$, let $v$ be a place of $k$, and write $k_v$ for the completion of
$k$ at $v$. We select an absolute value $\|\ \|_v$ from the place $v$ so that
\begin{itemize}
\item[(i)] if $v|\infty$ then $\|\ \|_v$ is the unique absolute value on $k_v$ that extends the usual absolute 
value on $\Q_{\infty} = \R$,
\item[(ii)] if $v|p$ then $\| \ \|_v$ is the unique absolute value on $k_v$ that extends the usual $p$-adic absolute value on $\Q_p$.
\end{itemize}
Let $Y$ denote the set of all places $y$ of the algebraically closed field $\oQ$, and assume that $k/\Q$ is a finite, Galois
extension.  At each place $v$ of $k$ we write
\begin{equation*}\label{un0}
Y(k,v) = \{y\in Y: y|v\}
\end{equation*}
for the subset of places in $Y$ that lie over $v$.  Clearly we can express $Y$ as the disjoint union
\begin{equation}\label{un1}
Y = \bigcup_v~Y(k,v),
\end{equation}
where the union is over all places $v$ of $k$.  In \cite[section 2]{all2009} the authors show that each subset $Y(k, v)$
can be expressed as an inverse limit of finite sets.  This determines a totally disconnected, compact, Hausdorff topology
in $Y(k, v)$.  Then it follows from (\ref{un1}) that $Y$ is a totally disconnected, locally compact, Hausdorff space.  The
topology induced in $Y$ does not depend on the number field $k$.  It is also shown in \cite[section 3]{all2009} that
for each finite, Galois extension $k/\Q$, the Galois group $\Aut(\oQ/k)$ acts transitively and continuously on the 
elements of each compact, open subset $Y(k, v)$.  Moreover (see \cite[Theorem 4]{all2009}), there exists a regular measure
$\lambda$ defined on the Borel subsets $\B$ of $Y$, such that $\lambda$ is positive (or infinite)
on nonempty open sets, finite on compact sets, and satisfies the identity $\lambda(\tau E) = \lambda (E)$ for all automorphisms
$\tau$ in $\Aut(\oQ/k)$, and all Borel subsets $E$ of $Y$.  The measure $\lambda$ is unique up to a positive multiplicative
constant.  In \cite[Theorem 5]{all2009} it is shown that $\lambda$ can be normalized so that
\begin{equation}\label{un2}
\lambda\bigl(Y(k, v)\bigr) = \dfrac{[k_v: \Q_v]}{[k: \Q]}
\end{equation}
for each finite Galois extension $k/\Q$ and each place $v$ of $k$.  More generally, if $k/\Q$ is a finite,
but not necessarily Galois extension, then the identity (\ref{un2}) continues to hold.  This is proved in \cite[Lemma 6]{vaaler2012}.

If $y$ is a place in $Y(k,v)$, we select an absolute value $\|\ \|_y$ from $y$ such that the restriction of $\|\ \|_y$ to $k$ is equal 
to $\|\ \|_v$.  As the restriction of $\|\ \|_v$ to $\Q$ is one of the usual absolute values on $\Q$, it follows that this choice
of the normalized absolute value $\|\ \|_y$ does not depend on $k$.  If $\alpha$ is a point in $\G$, we 
associate $\alpha$ with the continuous, compactly supported function
\begin{equation}\label{ht11}
y\mapsto f_{\alpha}(y) = \log \|\alpha\|_y
\end{equation}
defined on the locally compact Hausdorff space $Y$, (see \cite[equation (1.9)]{all2009}).  Each function (\ref{ht11})
belongs to the real Banach space $L^1(Y, \B, \lambda)$, where $\B$ is the $\sigma$-algebra of Borel subsets of $Y$, and
$\lambda$ is the normalized measure on $\B$ that satifies (\ref{un2}), and is invariant with respect to the natural Galois 
action on each compact, open subset $Y(k, v)$ (see \cite[Theorem 4]{all2009}).  It follows that the map
\begin{equation}\label{ht12}
\alpha\mapsto f_{\alpha}
\end{equation}
is an injective, linear transformation from the $\Q$-vector space $\G$ into the real Banach space $L^1(Y, \B, \lambda)$.  Let
\begin{equation}\label{ht13}
\F = \big\{f_{\alpha}(y): \alpha\in \G\big\} \subseteq L^1(Y, \B, \lambda)
\end{equation}
denote the image of $\G$ under the linear map (\ref{ht12}).  Then $\F$ is a $\Q$-vector space, and each element 
of $\F$ is a continuous, compactly supported function
\begin{equation*}\label{ht14}
f_{\alpha} : Y \rightarrow \R.
\end{equation*}
The map $\alpha\mapsto 2h(\alpha)$ is a norm on the $\Q$-vector space $\G$, and
$f_{\alpha}\mapsto \|f_{\alpha}\|_1$ is obviously a norm on the $\Q$-vector space $\F$.  With respect to these norms, 
the map $\alpha \mapsto f_{\alpha}$ is a linear isometry from the vector space $\G$ (written multiplicatively) onto the vector 
space $\F$ (written additively).  This follows because (see \cite[equation (1.11)]{all2009})
\begin{equation}\label{ht15}
2 h(\alpha) = \int_Y \bigl|f_{\alpha}(y)\bigr|\ \dla(y) = \|f_{\alpha}\|_1
\end{equation}
at each point $\alpha$ in $\G$.  The product formula (see \cite[equation (1.10)]{all2009}), implies that each 
function $f_{\alpha}$ in $\F$ belongs to the closed, co-dimension one subspace
\begin{equation}\label{ht20}
\X = \bigg\{F\in L^1(Y, \B, \lambda): \int_{Y}F(y)\ \dla(y) = 0\bigg\}.
\end{equation}
Then \cite[Theorem 1]{all2009} asserts that $\F$ is a dense subset of $\X$.

Let $\sS \subseteq \F$ be a subset, and let $\closure \sS \subseteq \X$ be the closure of $\sS$ in the $L^1$-norm 
topology of $\X$.  We can also form the closure of $\sS$ in the $L^1$-norm topology of $\F$, and this is clearly the 
subset $\F \cap \closure \sS$.  We say that the subset $\sS \subseteq \F$ is {\it $\F$-closed} if
\begin{equation*}\label{ht25}
\sS = \F \cap \closure \sS.
\end{equation*}
Thus $\sS$ is $\F$-closed precisely when $\sS$ is closed as a subset of $\F$, where $\F$ is given the $L^1$-norm 
topology induced, using (\ref{ht15}), by the Weil height.

We now formulate an alternative statement of Theorem \ref{thmintro3}.  Let $K/\Q$ be an algebraic field extension, and let
\begin{equation*}\label{ht30}
\F_K = \big\{f_{\alpha}(y) : \alpha \in \G_K\big\}
\end{equation*}
be the image of $\G_K$ in the $\Q$-vector space $\F$.  Then write
\begin{equation}\label{ht35}
\D_K = \spann_{\Q} \F_K
\end{equation}
for the $\Q$-linear subspace of $\F$ generated (additively) by $\F_K$.  
Each element of $\D_K$ is a finite linear combination
\begin{equation}\label{ht37}
\sum_{n = 1}^N q_n f_{\eta_n}(y),
\end{equation}
where $q_1, q_2, \dots , q_N$, are rational numbers, and $\eta_1, \eta_2, \dots , \eta_N$, are elements of $\G_K$.
If the positive integer $m$ is the least common multiple of the denominators of $q_1, q_2, \dots , q_N$, then it is clear that
(\ref{ht37}) can be written more simply as
\begin{equation}\label{ht39}
m^{-1} f_{\beta}(y),
\end{equation}
with $\beta$ in $\G_K$.  That is, (\ref{ht39}) is a generic element of the $\Q$-vector space $\D_K$, a conclusion
that also follows from (\ref{intro55}).

Because the map (\ref{ht12}) is a linear isometry from $\G$, with metric induced by the norm $\alpha \mapsto 2h(\alpha)$, 
onto $\F$, with metric induced by the $L^1$-norm, Corollary \ref{thmintro3} asserts that $\D_K \subseteq \F$ is 
$\F$-closed.  The following result is a reformulation of Theorem \ref{thmintro3}.

\begin{theorem}\label{thmintro4}  Let $K_1, K_2, \dots , K_N$, be a collection of fields such that
\begin{equation*}\label{ht60}
\Q \subseteq K_n \subseteq \oQ,\quad\text{for each $n = 1, 2, \dots , N$},
\end{equation*}
and for each pair of integers $n_1 \not= n_2$, either
\begin{equation}\label{ht63}
K_{n_1}/(K_{n_1} \cap K_{n_2}) \quad\text{is a (possibly infinite) Galois extension,}
\end{equation}
or
\begin{equation}\label{ht66}
K_{n_2}/(K_{n_1} \cap K_{n_2}) \quad\text{is a (possibly infinite) Galois extension.}
\end{equation}
For each $n = 1, 2, \dots , N$, let
\begin{equation*}\label{ht69}
\D_{K_n} = \spann_{\Q} \F_{K_n}
\end{equation*}
be the $\Q$-linear subspace generated by $\F_{K_n}$.  Then the $\Q$-linear subspace
\begin{equation}\label{ht72}
\D_{K_1} + \D_{K_2} + \cdots + \D_{K_N}
\end{equation}
is $\F$-closed.
\end{theorem}

For each $n = 1, 2, \dots , N$, our proof of Theorem \ref{thmintro4} uses a system of continuous, surjective, $\Q$-linear 
projections
\begin{equation*}\label{ht79}
S_{K_n} : \F \rightarrow \D_{K_n}.
\end{equation*}
If for each $n = 1, 2, \dots , N$ we define
\begin{equation*}\label{ht83}
\E_{K_n} = \ker S_{K_n} = \big\{f_{\alpha} \in \F : S_{K_n}\bigl(f_{\alpha}\bigr) = 0\big\},
\end{equation*}
then each $\Q$-linear subspace $\E_{K_n} \subseteq \F$ is also $\F$-closed, and we have the collection of direct 
sum decompositions
\begin{equation*}\label{ht88}
\F = \D_{K_n} \oplus \E_{K_n}.
\end{equation*}
The hypotheses (\ref{ht63}) and (\ref{ht66}) imply that for $m \not= n$ the projections $S_{K_m}$ and 
$S_{K_n}$ commute.  This leads to the direct sum decomposition
\begin{equation*}\label{ht93}
\F = \bigl(\D_{K_1} + D_{K_2} + \cdots + \D_{K_N}\bigr) \oplus \bigl(\E_{K_1} \cap \E_{K_2} \cap \cdots \cap \E_{K_N}\bigr),
\end{equation*}
and to the conclusion that the $\Q$-linear subspace (\ref{ht72}) is $\F$-closed.

The $\Q$-linear subspaces $\E_{K_n}$ are also of arithmetical interest.  For $n = 1, 2, \dots , N$, let
\begin{equation*}\label{ht97}
T_{K_n} : \F \rightarrow \E_{K_n}
\end{equation*}
be the linear map defined by
\begin{equation*}\label{ht101}
T_{K_n}\bigl(f_{\alpha}\bigr) = f_{\alpha} - S_{K_n}\bigl(f_{\alpha}\bigr).
\end{equation*}
It follows in a standard manner that each map $T_{K_n}$ is a surjective, continuous, linear projection of $\F$ onto $\E_{K_n}$.
If $m \not= n$ then $S_{K_m}$ and $S_{K_n}$ commute, and this implies that $T_{K_m}$ and $T_{K_n}$ also commute.
The argument used to prove that (\ref{ht72}) is $\F$-closed can be applied to the projections $T_{K_n}$, and leads
to the conclusion that the $\Q$-linear subspace
\begin{equation*}\label{ht107}
\E_{K_1} + \E_{K_2} + \cdots + \E_{K_N}
\end{equation*}
is also $\F$-closed.  More generally, let $M$ be an integer such that $0 \le M \le N$.
Then an obvious modification to the argument outlined here shows that the $\Q$-linear subspace
\begin{equation*}\label{ht111}
\D_{K_1} + \D_{K_2} + \cdots + \D_{K_M} + \E_{K_{M+1}} + \E_{K_{M+2}} + \cdots + \E_{K_N}
\end{equation*}
is $\F$-closed.  We state and prove this more general result, which includes Theorem \ref{thmintro3} and 
Theorem \ref{thmintro4} as special cases, in Section 8.

\section{The Galois action on $\G$}\label{pl}

We assume that $K/\Q$ is an algebraic extension of fields, and we define
\begin{equation*}\label{new1}
\delta_K : \oq \rightarrow \{1, 2, \dots \}
\end{equation*}
by
\begin{equation}\label{new2}
\delta_K(\alpha) = \min\big\{\bigl[K\bigl(\alpha^m\bigr): K\bigr] : m \in \Z,\ m \not= 0\big\}.
\end{equation}

\begin{lemma}\label{lemnew1}  Let $\alpha$ be an element of $\oq$, $\zeta$ an element of $\ot$, and let $\ell \not= 0$ 
be an integer.  Then we have
\begin{equation}\label{new5}
\delta_K\bigl(\alpha^{\ell}\bigr) = \delta_K(\alpha) = \delta_K(\alpha \zeta).
\end{equation}
\end{lemma}

\begin{proof}  For each $\alpha$ in $\oq$ there exists a smallest positive integer $n_1$ such that
\begin{equation*}\label{new10}
\delta_K(\alpha) = \bigl[K\bigl(\alpha^{n_1}\bigr) : K\bigr].
\end{equation*}
Hence we get
\begin{equation*}\label{new15}
\delta_K(\alpha) \le \bigl[K\bigl(\alpha^{m_1 n_1}\bigr) : K\bigr]
\end{equation*}
for all integers $m_1 \not= 0$.  However, for each integer $m_1 \not= 0$ the algebraic number 
$\alpha^{m_1 n_1}$ belongs to the field $K\bigl(\alpha^{n_1}\bigr)$.  It follows that
\begin{equation*}\label{new20}
\delta_K(\alpha) \le \bigl[K\bigl(\alpha^{m_1 n_1}\bigr) : K\bigr] \le \bigl[K\bigl(\alpha^{n_1}\bigr) : K\bigr] = \delta_K(\alpha),
\end{equation*}
and therefore
\begin{equation}\label{new25}
\delta_K(\alpha) = \bigl[K\bigl(\alpha^{m_1 n_1}\bigr) : K\bigr]
\end{equation}
for all integers $m_1 \not= 0$.  In a similar manner, if $n_2$ is the smallest positive integer such that
\begin{equation*}\label{new30}
\delta_K\bigl(\alpha^{\ell}\bigr) = \bigl[K\bigl(\alpha^{\ell n_2}\bigr) : K\bigr],
\end{equation*}
we find that
\begin{equation}\label{new35}
\delta_K\bigl(\alpha^{\ell}\bigr) = \bigl[K\bigl(\alpha^{\ell m_2 n_2}\bigr) : K\bigr]
\end{equation}
for all integers $m_2 \not= 0$.  We select $m_1 = \ell n_2$ and $m_2 = n_1$, so that
\begin{equation*}\label{new40}
m_1 n_1 = \ell m_2 n_2.
\end{equation*}
Combining (\ref{new25}) and (\ref{new35}) we get
\begin{equation*}\label{new45}
\delta_K(\alpha) = \bigl[K\bigl(\alpha^{m_1 n_1}\bigr) : K\bigr] 
					= \bigl[K\bigl(\alpha^{\ell m_2 n_2}\bigr) : K\bigr] = \delta_K\bigl(\alpha^{\ell}\bigr),
\end{equation*}
and this verifies the equality on the left of (\ref{new5}).

We have (\ref{new25}) for all integers $m_1 \not= 0$, but now we select $m_1 \not= 0$ so that $\zeta^{m_1} = 1$.  
We find that
\begin{equation}\label{new50}
\begin{split}
\delta_K(\alpha \zeta) &= \min\big\{\bigl[K\bigl((\alpha \zeta)^n\bigr): K\bigr] : n \in \Z,\ n \not= 0\big\}\\
                                        &\le \bigl[K\bigl((\alpha \zeta)^{m_1 n_1}\bigr): K\bigr]\\
                                        &= \bigl[K\bigl(\alpha^{m_1 n_1}\bigr): K\bigr]\\
                                        &= \delta_K(\alpha).
\end{split}
\end{equation}
As $\zeta^{-1}$ belongs to $\ot$, we also get
\begin{equation}\label{new55}
\delta_K(\alpha) = \delta_K\bigl((\alpha \zeta) \zeta^{-1}\bigr) \le \delta_K(\alpha \zeta).
\end{equation}
Plainly (\ref{new50}) and (\ref{new55}) confirm the equality on the right of (\ref{new5}).
\end{proof}

It follows from the equality on the right of (\ref{new5}) that $\delta_K$ is constant on cosets of the quotient group 
$\G$ defined in (\ref{intro17}).  Therefore $\delta_K$ is well defined as a map
\begin{equation*}\label{new57}
\delta_K : \G \rightarrow \{1, 2, \dots \}.
\end{equation*}
Because $\G$ is a $\Q$-vector space with scalar multiplication defined by (\ref{intro30}), the equality
on the left of (\ref{new5}) implies that
\begin{equation*}\label{new59}
\delta_K\bigl(\alpha^{r/s}\bigr) = \delta_K\bigl(\alpha^r\bigr) = \delta_K(\alpha).
\end{equation*}
Therefore $\delta_K$ is constant on $\Q$-linear subspaces of $\G$ having dimension $1$.

The profinite group $\Aut(\oQ/K)$ acts on elements of $\oq$, and $\Aut(\oQ/K)$ acts on the torsion subgroup
$\ot$.  Hence we get an action of $\Aut(\oQ/K)$ on the quotient group $\G$.  If $\alpha$ and 
$\beta$ in $\oq$ represent the same coset in $\G$, and $\tau$ belongs to $\Aut(\oQ/K)$, then $\tau \alpha$ and 
$\tau \beta$ represent the same coset in $\G$.  The remaining requirements for a group action are easily verified.
We find that for each $\tau$ in $\Aut(\oQ/K)$, and $r/s$ in $\Q$, we have
\begin{equation*}\label{new61}
\tau\bigl(\alpha^{r/s}\bigr) = \bigl(\tau \alpha\bigr)^{r/s}.
\end{equation*}
Therefore each automorphism $\tau$ in $\Aut(\oQ/K)$ acts as a linear transformation (written multiplicatively)
on the $\Q$-vector space $\G$.

If $\Aut(\oQ/K)$ acts on $\alpha$, where $\alpha$ is an element of $\oq$, then each orbit in $\oq$ has exactly 
$[K(\alpha) : K]$ distinct elements.   If $\Aut(\oQ/K)$ acts on $\alpha$, where $\alpha$ is a coset representative in
$\G$, then the orbit in $\G$ may contain fewer than $[K(\alpha) : K]$ distinct coset representatives.  We write
\begin{equation}\label{new65}
\Orb_K(\alpha) = \{\tau \alpha : \tau \in \Aut(\oQ/K)\} \subseteq \G
\end{equation}
for the set of coset representatives in the orbit of $\alpha$ under the action of $\Aut(\oQ/K)$ on $\G$.
We write $|\Orb_K(\alpha)|$ for the number of distinct coset representatives in the set (\ref{new65}), so that
\begin{equation}\label{new70}
|\Orb_K(\alpha)| \le [K(\alpha) : K].
\end{equation}

\begin{lemma}\label{lemnew2}  Let $\alpha$ be an element of $\G$.  Then we have
\begin{equation*}\label{new80}
|\Orb_K(\alpha)| = \delta_K(\alpha).
\end{equation*}
\end{lemma}

\begin{proof}  As $\G$ is a $\Q$-vector space, and each $\tau$ in $\Aut(\oQ/K)$ acts on $\G$ as a linear transformation, the map
\begin{equation}\label{new90}
\ell \mapsto \bigl|\Orb_K\bigl(\alpha^{\ell}\bigr)\bigr|
\end{equation}
is constant on the set of integers $\ell \not= 0$.  It follows using (\ref{new70}) and (\ref{new90}) that
\begin{equation}\label{new95}
|\Orb_K(\alpha)| \le \min\big\{\bigl[K\bigl(\alpha^{\ell}\bigr) : K\bigr] : \text{$\ell \in \Z$, $\ell \not= 0$}\big\} = \delta_K(\alpha).
\end{equation}

Let $\alpha$ in $\oq$ represent a coset in $\G$, and then let $n$ be the smallest positive integer such that
\begin{equation*}\label{new98}
\delta_K(\alpha) = \bigl[K\bigl(\alpha^n\bigr) : K\bigr] = L.
\end{equation*}
As in our proof of Lemma \ref{lemnew1}, we have
\begin{equation*}\label{new100}
\delta_K(\alpha) = \bigl[K\bigl(\alpha^{m n}\bigr) : K\bigr] = L
\end{equation*}
for all integers $m \not= 0$.  Write
\begin{equation}\label{new103}
\alpha^n = \alpha_1^n, \alpha_2^n, \dots , \alpha_L^n
\end{equation}
for the distinct conjugates of $\alpha^n$ over the field $K$.  Then for each integer $m \not= 0$ the algebraic numbers
\begin{equation}\label{new105}
\alpha^{m n} = \alpha_1^{m n}, \alpha_2^{m n}, \dots , \alpha_L^{m n}
\end{equation}
are the distinct conjugates of $\alpha^{m n}$ over $K$.  We claim that the $L$ algebraic numbers (\ref{new103}) are
distinct coset representatives in $\G$.  If this is not the case, then there exist $1 \le i < j \le L$ and $\zeta$ in 
$\ot$, such that
\begin{equation}\label{new110}
\alpha_i^n = \alpha_j^n \zeta.
\end{equation}
Let $m$ be a positive integer such that $\zeta^m = 1$.  Then (\ref{new110}) implies that
\begin{equation*}\label{new115}
\alpha_i^{m n} = \bigl(\alpha_j^n \zeta\bigr)^m = \alpha_j^{m n},
\end{equation*}
and this contradicts the fact that the algebraic numbers (\ref{new105}) are distinct conjugates over $K$.  Thus
our claim that the numbers (\ref{new103}) are distinct coset representatives in $\G$ is verified.  We conclude that
\begin{equation}\label{new120}
\delta_K(\alpha) = L \le \bigl|\Orb\bigl(\alpha^n\bigr)\bigr| = |\Orb_K(\alpha)|.
\end{equation}
The lemma follows from (\ref{new95}) and (\ref{new120}).
\end{proof}

From the definition (\ref{new2}) we have $\delta_K(\alpha) = 1$ if and only if $\alpha^m$ belongs to $K$ for some 
integer $m \not= 0$.  Using (\ref{intro55}) we find that $\delta_K(\alpha) = 1$ if and only if $\alpha$ belongs to 
$\spann_{\Q} \G_K$.  Then it follows from (\ref{intro55}) and Lemma \ref{lemnew2}, that 
\begin{equation}\label{new123}
\begin{split}
\spann_{\Q} \G_K &= \big\{\alpha \in \G : \bigl|\Orb_K(\alpha)\bigr| = 1\big\}\\
			    &= \{\alpha \in \G : \text{$\tau \alpha = \alpha$ for each $\tau$ in $\Aut(\oQ/K)$}\}.
\end{split}
\end{equation}
Suppose, more generally, that $\beta$ belongs to $\G$, and
\begin{equation}\label{new125}
\Orb_K(\beta) = \big\{\beta_1, \beta_2, \dots , \beta_L\big\},
\end{equation}
where $\delta_K(\beta) = L$.  Write
\begin{equation*}\label{new127}
\gamma = \beta_1 \beta_2 \cdots \beta_L.
\end{equation*}
As each $\tau$ in $\Aut(\oQ/K)$ permutes the distinct elements of $\Orb_K(\beta)$, we find that $\tau \gamma = \gamma$
for each $\tau$ in $\Aut(\oQ/K)$.  We conclude from (\ref{new123}) that $\gamma$ belongs to $\spann_{\Q} \G_K$.

\section{Commuting projections}\label{dsf}

In this section we work in the $\Q$-vector space $\F$ defined in (\ref{ht13}).  We consider
subspaces of $\F$ which are not necessarily associated to a field extension $K/\Q$.

Let $\hH \subseteq \F$ and $\I \subseteq \F$ be $\Q$-linear subspaces such that
\begin{equation*}\label{form1}
\F = \hH \oplus \I.
\end{equation*}
Then each element $f_{\alpha}$ in $\F$ has a unique representation as
\begin{equation*}\label{form5}
f_{\alpha}(y) = f_{\beta}(y) + f_{\gamma}(y),
\end{equation*}
where $f_{\beta}$ belongs to $\hH$, and $f_{\gamma}$ belongs to $\I$. 
Let $S : \F \rightarrow \hH$ be the surjective, linear projection defined by
\begin{equation*}\label{form12}
S\bigl(f_{\alpha}\bigr) = S\bigl(f_{\beta} + f_{\gamma}\bigr) = f_{\beta}.
\end{equation*}
It follows that $\I = \ker S$.

\begin{lemma}\label{lemform1}  Let $\hH_1, \hH_2, \I_1$, and $\I_2$, be $\Q$-linear subspaces of $\F$ such that
\begin{equation*}\label{form101}
\F = \hH_1 \oplus \I_1 = \hH_2 \oplus \I_2.
\end{equation*}
Let $S_1 : \F \rightarrow \hH_1$ and $S_2 : \F \rightarrow \hH_2$ be the corresponding surjective, linear projections
such that $\I_1 = \ker S_1$ and $\I_2 = \ker S_2$.  If $S_1$ and $S_2$ commute, then
\begin{equation}\label{form103}
\F = (\hH_1 + \hH_2) \oplus (\I_1 \cap \I_2),
\end{equation}
and the surjective linear projection $W_2 : \F \rightarrow \hH_1 + \hH_2$ such that
\begin{equation*}\label{form103.5}
\iim W_2 = \hH_1 + \hH_2,\quad\text{and}\quad \ker W_2 = \I_1 \cap \I_2,
\end{equation*}
is given by
\begin{equation}\label{form104}
W_2 = S_1 + S_2 - S_1\circ S_2.
\end{equation}
Moreover, if both $S_1$ and $S_2$ are continuous, then $W_2$ is continuous.
\end{lemma}

\begin{proof}  If $S_1$ and $S_2$ commute, then it is easy to verify that the linear transformation $W_2$ defined by
(\ref{form104}) satisfies $W_2^2 = W_2$.  Therefore $W_2$ {\it is} a linear projection, and it follows that
\begin{equation*}\label{from105}
\F = \iim W_2 \oplus \ker W_2.
\end{equation*}
It remains to identify the image and kernel of $W_2$.

If $f_{\beta_1}$ belongs to $\hH_1$ and $f_{\beta_2}$ belongs to $\hH_2$, we find that
\begin{align*}\label{form106}
\begin{split}
W_2\bigl(f_{\beta_1}  + f_{\beta_2}\bigr) 
   &= f_{\beta_1} + S_1\bigl(f_{\beta_2}\bigr) + S_2\bigl(f_{\beta_1}\bigr) 
   				+ f_{\beta_2} - S_2\bigl(f_{\beta_1}\bigr) - S_1\bigl(f_{\beta_2}\bigr)\\
   &=	f_{\beta_1} + f_{\beta_2}.	 
\end{split}
\end{align*}
It follows that
\begin{equation}\label{form107}
\hH_1 + \hH_2 \subseteq \iim W_2.
\end{equation}

Let $f_{\alpha}$ be an element of $\F$ such that
\begin{equation}\label{form109}
f_{\alpha} = f_{\beta_1} + f_{\gamma_1},\quad\text{where}\ f_{\beta_1} \in \hH_1,\ \text{and}\ f_{\gamma_1} \in \I_1,
\end{equation}
and
\begin{equation}\label{form111}
f_{\alpha} = f_{\beta_2} + f_{\gamma_2},\quad\text{where}\ f_{\beta_2} \in \hH_2,\ \text{and}\ f_{\gamma_2} \in \I_2.
\end{equation}
We find that
\begin{align}\label{form113}
\begin{split}
W_2\bigl(f_{\alpha}\bigr) &= S_1\bigl(f_{\alpha}\bigr) + S_2\bigl(f_{\alpha}\bigr) - S_1\bigl(S_2\bigl(f_{\alpha}\bigr)\bigr)\\
	&= f_{\beta_1} + f_{\beta_2} - S_1\bigl(f_{\beta_2}\bigr),
\end{split}
\end{align}
and because $S_1$ and $S_2$ commute, we also get
\begin{align}\label{form115}
\begin{split}
W_2\bigl(f_{\alpha}\bigr) &= S_1\bigl(f_{\alpha}\bigr) + S_2\bigl(f_{\alpha}\bigr) - S_2\bigl(S_1\bigl(f_{\alpha}\bigr)\bigr)\\
	&= f_{\beta_1} + f_{\beta_2} - S_2\bigl(f_{\beta_1}\bigr).
\end{split}
\end{align}
Both (\ref{form113}) and (\ref{form115}) show that $W_2\bigl(f_{\alpha}\bigr)$ belongs to the subspace $\hH_1 + \hH_2$, 
and therefore
\begin{equation}\label{form117}
\iim W_2 \subseteq \hH_1 + \hH_2.
\end{equation}
Now (\ref{form107}) and (\ref{form117}) imply that
\begin{equation*}\label{form119}
\iim W_2 = \hH_1 + \hH_2.
\end{equation*}

We continue to assume that $f_{\alpha}$ is given by (\ref{form109}) and by (\ref{form111}). 
Let $T_1 : \F \rightarrow \I_1$ and $T_2 : \F \rightarrow \I_2$ be the surjective linear projections defined by
\begin{equation*}\label{form125}
T_1\bigl(f_{\alpha}\bigr) = f_{\alpha} - S_1\bigl(f_{\alpha}\bigr) = f_{\gamma_1},
			\quad\text{and}\quad T_2\bigl(f_{\alpha}\bigr) = f_{\alpha} - S_2\bigl(f_{\alpha}\bigr) = f_{\gamma_2}.
\end{equation*}
Using (\ref{form113}) we find that
\begin{equation}\label{form127}
f_{\alpha} - W_2\bigl(f_{\alpha}\bigr) = f_{\gamma_1} - f_{\beta_2} + S_1\bigl(f_{\beta_2}\bigr) 
				= f_{\gamma_1} - T_1\bigl(f_{\beta_2}\bigr)
\end{equation}
belongs to $\I_1$, and
\begin{equation}\label{form129}
f_{\alpha} - W_2\bigl(f_{\alpha}\bigr) = f_{\gamma_2} - f_{\beta_1} + S_2\bigl(f_{\beta_1}\bigr) 
				= f_{\gamma_2} - T_2\bigl(f_{\beta_1}\bigr)
\end{equation}
belongs to $\I_2$.  If $W_2\bigl(f_{\alpha}\bigr) = 0$, then (\ref{form127}) and (\ref{form129}) show that
\begin{equation}\label{form131}
\ker W_2 \subseteq \I_1 \cap \I_2.
\end{equation}
On the other hand, if $f_{\alpha}$ belongs to $\I_1 \cap \I_2$, then $f_{\beta_1} = f_{\beta_2} = 0$, and both
(\ref{form127}) and (\ref{form129}) imply that $W_2\bigl(f_{\alpha}\bigr) = 0$.  Thus we get
\begin{equation}\label{form133}
\I_1 \cap \I_2 \subseteq \ker W_2.
\end{equation}
From (\ref{form131}) and (\ref{form133}) we conclude that
\begin{equation*}\label{form135}
\ker W_2 = \I_1 \cap \I_2.
\end{equation*}
This completes the proof that (\ref{form103}) holds.

The last assertion of the lemma is obvious.
\end{proof}

\begin{theorem}\label{thmform1}  Let
\begin{equation*}\label{form137}
\hH_1, \hH_2, \dots , \hH_N,\quad\text{and}\quad \I_1, \I_2, \dots , \I_N,
\end{equation*}
be a collection of $\Q$-linear subspaces of $\F$, that satisfy
\begin{equation*}\label{form139}
\F = \hH_n \oplus \I_n,\quad\text{for each $n = 1, 2, \dots , N$.}
\end{equation*}
For each $n = 1, 2, \dots , N$, let $S_n : \F \rightarrow \hH_n$ be the unique, surjective linear projection such that 
$\I_n = \ker S_n$.  Assume that for $m \not= n$, the linear projections $S_m$ and $S_n$ commute.  Then we have
\begin{equation}\label{form145}
\F = (\hH_1 + \hH_2 + \cdots + \hH_N) \oplus (\I_1 \cap \I_2 \cap \cdots \cap \I_N),
\end{equation}
and the unique, surjective linear projection
\begin{equation*}\label{form147}
W_N : \F \rightarrow \hH_1 + \hH_2 + \cdots + \hH_N
\end{equation*}
such that 
\begin{equation*}\label{form149}
\ker W_N = \I_1 \cap \I_2 \cap \cdots \cap \I_N,
\end{equation*}
is given by
\begin{equation}\label{form151}
W_N = I - (I - S_1)\circ(I - S_2)\circ \cdots \circ (I - S_N),
\end{equation}
where $I : \F \rightarrow \F$ is the identity transformation.  Moreover, if each of the projections $S_1, S_2, \dots , S_N$, is
continuous, then $W_N$ is continuous.
\end{theorem}

\begin{proof}  We argue by induction on $N$.  If $N = 2$ then the result is exactly the statement of Lemma \ref{lemform1},
which has already been proved.  Therefore we assume that the hypotheses and conclusion of Theorem \ref{thmform1} hold
with $2 \le N$, and we assume that $\hH_{N+1}$ and $\I_{N+1}$ are $\Q$-linear subspaces of $\F$ such that
\begin{equation}\label{form155}
\F = \hH_{N+1} \oplus \I_{N+1}.
\end{equation}  
We write $S_{N+1} : \F \rightarrow \hH_{N+1}$ for the unique, surjective linear projection such that $\ker S_{N+1} = \I_{N+1}$,
and we assume that for $1\le n \le N$, the linear projections $S_n$ and $S_{N+1}$ commute.

Let $\hH_0$ and $\I_0$ be the $\Q$-linear subspaces of $\F$ given by
\begin{equation*}\label{form157}
\hH_0 = \hH_1 + \hH_2 + \cdots + \hH_N,\quad\text{and}\quad \I_0 = \I_1 \cap \I_2 \cap \cdots \cap \I_N.
\end{equation*}
By the inductive hypothesis 
\begin{equation}\label{form159}
\F = \hH_0 \oplus \I_0,
\end{equation}
and $W_N : \F \rightarrow \hH_0$ is the unique, surjective, linear projection such that $\ker W_N = \I_0$.  For $1 \le n \le N$
the projections $S_n$ and $S_{N+1}$ commute, and therefore the projections $(I - S_n)$ and $S_{N+1}$ also commute.
It follows that
\begin{align*}\label{form161}
\begin{split}
S_{N+1}\circ W_N &= S_{N+1} - S_{N+1}\circ(I - S_1)\circ(I - S_2)\circ \cdots \circ (I - S_N)\\
                           &= S_{N+1} - (I - S_1)\circ(I - S_2)\circ \cdots \circ (I - S_N)\circ S_{N+1}\\
                           &= W_N\circ S_{N+1},
\end{split}
\end{align*}
and this shows that the linear projections $W_N$ and $S_{N+1}$ commute.  Hence we may apply 
Lemma \ref{lemform1} to the two direct sum decompositions (\ref{form155}) and (\ref{form159}).  We conclude that 
\begin{equation*}\label{from163}
\begin{split}
\F &= (\hH_0 + \hH_{N+1}) \oplus (\I_0 \cap \I_{N+1})\\
    &= (\hH_1 + \hH_2 + \cdots + \hH_N + \hH_{N+1}) \oplus (\I_1 \cap \I_2 \cap \cdots \cap \I_N \cap \I_{N+1}).
\end{split}
\end{equation*}
And the unique, surjective, linear projection 
\begin{equation*}\label{form165}
W_{N+1} : \F \rightarrow \hH_1 + \hH_2 + \cdots + \hH_N + \hH_{N+1}
\end{equation*}
such that
\begin{equation*}\label{form167}
\ker W_{N+1} = \I_1 \cap \I_2 \cap \cdots \cap \I_N \cap \I_{N+1},
\end{equation*}
is given by
\begin{align}\label{form169}
\begin{split}
W_{N+1} &= W_N + S_{N+1} - W_N \circ S_{N+1}\\
                &= I - (I - S_1)\circ(I - S_2)\circ \cdots \circ (I - S_N)\\
                &\qquad\qquad - (I - S_1)\circ(I - S_2)\circ \cdots \circ (I - S_N)\circ S_{N+1}\\
                 &= I - (I - S_1)\circ(I - S_2)\circ \cdots \circ (I - S_N)\circ(I - S_{N+1}).
\end{split}
\end{align}
If each surjective, linear projection $S_n,$ for $1 \le n \le N+1$, is continuous, then it is obvious from (\ref{form169}) that 
$W_{N+1}$ is continuous.  This establishes the theorem with $N$ replaced by $N+1$, and completes the proof of
(\ref{form145}) and (\ref{form151}) by induction.

Again the last statement of the theorem is obvious.
\end{proof}

\section{The image of a continuous linear projection is $\F$-closed}\label{dsx}

Again we suppose that $\hH \subseteq \F$ and $\I \subseteq \F$ are $\Q$-linear subspaces such that
\begin{equation*}\label{form26}
\F = \hH \oplus \I.
\end{equation*}
The closure of $\hH$ in the $L^1$-norm topology of $\X$ is a closed 
$\R$-linear subspace of $\X$, and similarly for the closure of $\I$ in $L^1$-norm.  We write
\begin{equation*}\label{form28}
\closure \hH \subseteq \X,\quad\text{and}\quad \closure \I \subseteq \X,
\end{equation*}
for the closure of $\hH$ and the closure of $\I$, respectively, in $\X$.

\begin{theorem}\label{thmsum1}  Let $\hH \subseteq \F$ and $\I \subseteq \F$ be $\Q$-linear subspaces such that
\begin{equation*}\label{form33}
\F = \hH \oplus \I,
\end{equation*}
and let $S:\F \rightarrow \hH$ be the surjective, linear, projection such that $\I = \ker S$.  Then the 
following conditions are equivalent:
\begin{itemize}
\item[(i)]  the linear projection $S$ is continuous,
\item[(ii)]  the $\R$-linear subspaces $\closure \hH \subseteq \X$, and $\closure \I \subseteq \X$, satisfy
\begin{equation}\label{form35}
\X = \closure \hH \oplus \closure \I.
\end{equation}
\end{itemize}
\end{theorem}

\begin{proof}  Assume that $S$ is continuous.  As $\F \subseteq \X$ is dense, the map $S$ has a unique extension to a continuous 
linear map $\wS : \X \rightarrow \closure \hH$.  Let $F$ belong to $\X$, so that $\wS(F)$ belongs to $\closure \hH$, and let 
\begin{equation*}\label{form37}
\big\{f_{\beta_n} : n = 1, 2, \dots \big\}
\end{equation*}
be a sequence in $\hH$ that converges in $L^1$-norm to $\wS(F)$.  Because $\wS$ is continuous and extends $S$, we get
\begin{align*}\label{form39}
\begin{split}
\wS^2(F) &= \wS\bigl(\lim_{n\rightarrow \infty} f_{\beta_n}\bigr) = \lim_{n\rightarrow \infty}\wS\bigl(f_{\beta_n}\bigr)\\ 
	        &= \lim_{n\rightarrow \infty} S\bigl(f_{\beta_n}\bigr) = \lim_{n\rightarrow \infty} f_{\beta_n}\\
	        &= \wS(F).
\end{split}
\end{align*}
It follows that $\wS : \X \rightarrow \closure \hH$ is a continuous, linear projection.  As the image of a continuous, linear
projection is a closed subspace, we find that the image of $\wS$ is a closed subspace containing $\hH$ and contained
in $\closure \hH$.  We conclude that $\wS : \X \rightarrow \closure \hH$ is surjective, and therefore we get
\begin{equation}\label{form41}
\X = \closure \hH \oplus \ker \wS.
\end{equation}

Since $\wS$ is continuous,
\begin{equation*}\label{form43}
\ker \wS = \{F \in \X : \wS(F) = 0\}
\end{equation*}
is obviously a closed linear subspace that contains $\ker S = \I$.  Thus we have
\begin{equation}\label{form45}
\closure \I \subseteq \ker \wS.
\end{equation} 
Suppose that $F$ belongs to $\ker \wS$, and let
\begin{equation*}\label{form47}
\big\{f_{\alpha_n} : n = 1, 2, \dots \big\}
\end{equation*} 
be a sequence in $\F$ that converges in $L^1$-norm to $F$.  Write
\begin{equation*}\label{form49}
f_{\alpha_n} = f_{\beta_n} + f_{\gamma_n},\quad\text{where}\ f_{\beta_n} \in \hH,\ \text{and}\ f_{\gamma_n} \in \I,
\end{equation*}
for each $n = 1, 2, \dots  $.  Then we have
\begin{equation*}\label{form51}
0 = \wS(F) = \lim_{n\rightarrow \infty} \wS\bigl(f_{\alpha_n}\bigr)
    = \lim_{n\rightarrow \infty} S\bigl(f_{\alpha_n}\bigr) = \lim_{n\rightarrow \infty} f_{\beta_n},
\end{equation*}
and it follows that
\begin{equation*}\label{form53}
\lim_{n\rightarrow} f_{\gamma_n} = F.
\end{equation*}
That is, $F$ belongs to $\closure \I$, and therefore
\begin{equation}\label{form55}
\ker \wS \subseteq \closure \I.
\end{equation}
Now (\ref{form41}), (\ref{form45}), and (\ref{form55}), establish the identity (\ref{form35}).  We have proved that (i) implies (ii).

Next we assume that (\ref{form35}) holds.  Then it follows from the general theory of complemented subspaces in a Banach
space (see \cite[Theorem 3.2.14]{megginson}), that there exists a continuous, surjective, linear projection 
\begin{equation*}\label{form57}
U : \X \rightarrow \closure \hH
\end{equation*}
such that $\ker U = \closure \I$.  Moreover, if $F$ belongs to $\X$, if the unique decomposition 
of $F$ is given by
\begin{equation*}\label{form59}
F = F_1 + F_2,\quad\text{where}\ F_1 \in \closure \hH,\ \text{and}\ F_2 \in \closure \I,
\end{equation*} 
then
\begin{equation*}\label{form61}
U(F) = U\bigl(F_1 + F_2\bigr) = F_1.
\end{equation*}
In particular, if $f_{\alpha}$ belongs to the dense $\Q$-linear space $\F$, and
\begin{equation*}\label{form63}
f_{\alpha_n} = f_{\beta_n} + f_{\gamma_n},\quad\text{where}\ f_{\beta_n} \in \hH,\ \text{and}\ f_{\gamma_n} \in \I,
\end{equation*}
then it follows that
\begin{equation*}\label{form65}
U\bigl(f_{\alpha}\bigr) = U\bigl(f_{\beta} + f_{\gamma}\bigr) = f_{\beta}.
\end{equation*}
This shows that the restriction of the continuous, linear projection $U$ to $\F$ is equal to $S$.  Hence $S$ is continuous,
and it follows that (ii) implies (i).
\end{proof}

\begin{theorem}\label{thmsum2}  Let $\hH \subseteq \F$ and $\I \subseteq \F$ be $\Q$-linear subspaces such that
\begin{equation*}\label{form67}
\F = \hH \oplus \I,
\end{equation*}
and let $S:\F \rightarrow \hH$ be the surjective, linear projection such that $\I = \ker S$.  If $S$ is continuous, then 
the $\Q$-linear subspaces $\hH \subseteq \F$ and $\I \subseteq \F$ are both $\F$-closed.
\end{theorem}

\begin{proof}   Let $f_{\alpha}$ be a function in $\F \cap \closure \hH$, and let 
\begin{equation*}\label{form75}
\big\{f_{\beta_n} : n = 1, 2, \dots \big\}
\end{equation*}  
be a sequence of functions in $\hH$ such that
\begin{equation*}\label{from77}
\lim_{N \rightarrow \infty} f_{\beta_n} = f_{\alpha}
\end{equation*}
in $L^1$-norm.  Because $S$ is continuous and the restriction of $S$ to $\hH$ is the identity, we get  
\begin{align}\label{form79}
\begin{split}
\big\|f_{\alpha} - S\bigl(f_{\alpha}\bigr)\big\|_1 &\le \big\|f_{\alpha} - f_{\beta_n}\big\|_1 + \big\|f_{\beta_n} - S\bigl(f_{\alpha}\bigr)\big\|_1\\
	&= \big\|f_{\alpha} - f_{\beta_n}\big\|_1 + \big\|S\bigl(f_{\beta_n} - f_{\alpha}\bigr)\big\|_1\\
	&\le \bigl(1 + \|S\|\bigr) \big\|f_{\beta_n} - f_{\alpha}\big\|_1
\end{split}
\end{align}
for each positive integer $n$.  Letting $n \rightarrow \infty$ on the right of (\ref{form79}), we conclude that
\begin{equation*}\label{form81}
f_{\alpha}  = S\bigl(f_{\alpha}\bigr).
\end{equation*}
As $S$ maps $\F$ onto $\hH$, we find that $f_{\alpha}$ belongs to the subspace $\hH$.  We have shown that
\begin{equation*}\label{form83}
\F \cap \closure \hH \subseteq \hH,
\end{equation*}
and the reverse containment is obvious.   This verifies that $\hH$ is $\F$-closed.

Let $T \rightarrow \I$ be defined by
\begin{equation*}\label{form85}
T\bigl(f_{\alpha}\bigr) = f_{\alpha} - S\bigl(f_{\alpha}\bigr).
\end{equation*}
Then $T$ is a continuous, linear projection onto $\I$, and the identity
\begin{equation*}\label{form90}
\I = \F \cap \closure \I
\end{equation*}
can be established by a similar argument using $T$.  Thus $\I$ is $\F$-closed.
\end{proof}

\section{Projections onto subspaces generated by fields}\label{comp} 

A basic problem for an infinite dimensional Banach space is to identify those closed linear subspaces that are
complemented.  We recall (see \cite{megginson} for further details) that a linear operator $U : \X \rightarrow \X$ is 
a {\it projection} if $U^2 = U$.  Then the image of a continuous, linear projection is a closed linear subspace of $\X$.  We say 
that a closed, linear subspace $\U \subseteq \X$ is {\it complemented} in $\X$ if and only if $\U$ is the image of a 
continuous, linear projection $U : \X \rightarrow \U$.  Alternatively, $\U$ is complemented in $\X$ if and only if there exists 
a second closed linear subspace $\V \subseteq \X$ such that $\X$ has the direct sum decomposition
\begin{equation*}\label{form96}
\X = \U \oplus \V.
\end{equation*}
If $\U$ is the image of the continuous, linear projection $U$, then we can take $\V$ to be the kernel of $U$.  
In this section we will show that if $K/\Q$ is an algebraic extension of fields, then the closed, $\R$-linear subspace 
\begin{equation*}\label{form98}
\X_K = \closure \D_K 
\end{equation*}
is complemented in $\X$.  

Let $u$ be a place of $\Q$.  As discussed in \cite[section 3]{all2009}, the profinite group $\Aut(\oQ/\Q)$ acts transitively 
and continuously on each compact, open subset
\begin{equation*}\label{form100}
Y(\Q, u) = \big\{y \in Y :  y|u\big\}.
\end{equation*}
In particular, it follows from \cite[Lemma 3]{all2009} that the map
\begin{equation*}\label{form102}
(\tau, y) \mapsto \tau y
\end{equation*}
from $\Aut(\oQ/\Q) \times Y(\Q, u)$ onto $Y(\Q, u)$ is continuous.   As $Y$ is the disjoint union of the compact, open sets
\begin{equation*}\label{form105}
Y = \bigcup_u Y(\Q, u),
\end{equation*} 
we conclude that $(\tau, y) \mapsto \tau y$ is a continuous map from $\Aut(\oQ/\Q) \times Y$ onto $Y$.  Then it is obvious that
\begin{equation}\label{form110}
(\tau, y) \mapsto \tau^{-1} y
\end{equation}
is also a continuous map from $\Aut(\oQ/\Q) \times Y$ onto $Y$.  

Suppose that $F$ belongs to $L^1(Y, \B, \lambda)$.  Because (\ref{form110}) is continuous, we find that for each 
$\tau$ in $\Aut(\oQ/\Q)$, the map
\begin{equation}\label{form116}
y \mapsto F(\tau^{-1} y)
\end{equation}
is a Borel measurable function from $Y$ into $\R$.  Then it follows from \cite[Theorem 4]{all2009} that for each 
$\tau$ in $\Aut(\oQ/\Q)$ the function (\ref{form116}) belongs to $L^1(Y, \B, \lambda)$, and
\begin{equation}\label{form120}
\int_Y |F(\tau^{-1} y)\bigr|\ \dla(y) = \int_Y |F(y)|\ \dla(y) = \|F\|_1.
\end{equation}
We use these observations to define a group of continuous, linear isometries.

For each automorphism $\tau$ in $\Aut(\oQ/\Q)$, we define a map
\begin{equation*}\label{form310}
\Phi_{\tau}: L^1(Y, \B, \lambda) \rightarrow L^1(Y, \B, \lambda)
\end{equation*}
by
\begin{equation}\label{form315}
\Phi_{\tau}(F)(y) = F(\tau^{-1} y).
\end{equation}
It is obvious that $\Phi_{\tau}$ is a linear map, and (\ref{form120}) shows that $\Phi_{\tau}(F)$ belongs to $L^1(Y, \B, \lambda)$.
Moreover, it follows from (\ref{form120}) that $\Phi_{\tau}$ is a linear isometry, and therefore $\Phi_{\tau}$ is continuous.  
(See \cite{lamperti1958} for a general representation of isometries on $L^p$-spaces.)  Applying 
\cite[Theorem 4]{all2009} again, we find that
\begin{equation}\label{form325}
\int_Y \Phi_{\tau}(F)(y)\ \dla(y) = \int_Y F(y)\ \dla(y).
\end{equation}
If $F$ belongs to the closed subspace $\X$, then it follows from (\ref{form325}) that $\Phi_{\tau}(F)$
belongs to $\X$.  Therefore $\Phi_{\tau}$ restricted to $\X$ is a linear isometry mapping $\X$ onto $\X$.  It will be 
convenient for our purposes to restrict the domain of each operator $\Phi_{\tau}$ to the subspace $\X$.
Thus for each $\tau$ in $\Aut(\oQ/\Q)$, we understand the map $\Phi_{\tau}$ to be a linear isometry
\begin{equation*}\label{form330}
\Phi_{\tau} : \X \rightarrow \X.
\end{equation*}

For a function $F$ in $\X$, and automorphisms $\sigma$ and $\tau$ in $\Aut(\oQ/\Q)$, we have
\begin{equation*}\label{form335}
\Phi_{\sigma}\bigl(\Phi_{\tau}(F)\bigr)(y) = \Phi_{\sigma\tau}(F)(y).
\end{equation*}
It follows that $\tau \mapsto \Phi_{\tau}$ is a homomorphism from the group $\Aut(\oQ/\Q)$ into the group 
$\Iso(\X)$ of {\it all} linear isometries of $\X$ onto itself.  The image of this homomorphism is obviously the subgroup
\begin{equation}\label{form340}
\big\{\Phi_{\tau} : \tau \in \Aut(\oQ/\Q)\big\} \subseteq \Iso(\X).
\end{equation}
If $f_{\alpha}(y) = \log \|\alpha\|_y$
belongs to the $\Q$-vector space $\F \subseteq \X$, then it follows from (\ref{form315}) that
\begin{equation}\label{form345}
\Phi_{\tau}\bigl(f_{\alpha}\bigr)(y) = f_{\alpha}\bigl(\tau^{-1}y\bigr) = \log \|\tau \alpha\|_y = f_{\tau \alpha}(y).
\end{equation}
Hence the maps (\ref{form340}), with domains restricted to the $\Q$-vector space $\F \subseteq \X$, act as a 
group of continuous, linear, isometries of $\F$ onto itself.  

Let $K/\Q$ is an algebraic extension of fields.  We consider the subgroup
\begin{equation}\label{form350}
\big\{\Phi_{\sigma} : \sigma \in \Aut(\oQ/K)\big\} \subseteq \big\{\Phi_{\tau} : \tau \in \Aut(\oQ/\Q)\big\}.
\end{equation}
Lemma \ref{lemnew2} and (\ref{form345}) imply that the orbit of each function $f_{\alpha}$ in $\F$ under the
action of the subgroup on the left of (\ref{form350}), contains
\begin{equation*}\label{form347}
\bigl|\Orb_K(\alpha)\bigr| = \delta_K(\alpha) = \delta_K\bigl(f_{\alpha}\bigr)
\end{equation*}
distinct functions.  In the following result we identify the set of fixed points when this subgroup acts on $\F$, and on $\X$.
We recall that $\D_K$, as defined in (\ref{ht35}), is the $\Q$-linear subspace of $\F$ generated by $\F_K$. 

\begin{lemma}\label{lemcomp1}  Let $K/\Q$ be an algebraic extension of fields.  Then 
\begin{equation}\label{form355}
\D_K = \big\{f_{\alpha} \in \F : \text{$\Phi_{\tau}\bigl(f_{\alpha}\bigr) = f_{\alpha}$ for each $\tau$ in $\Aut(\oQ/K)$}\big\},
\end{equation}
and
\begin{equation}\label{form360}
\closure \D_K = \big\{F \in \X : \text{$\Phi_{\tau}(F) = F$ for each $\tau$ in $\Aut(\oQ/K)$}\big\}.
\end{equation}
\end{lemma}

\begin{proof}  The identity (\ref{new123}) asserts that
\begin{equation}\label{form363}
\spann_{\Q} \G_K = \{\alpha \in \G : \text{$\tau \alpha = \alpha$ for each $\tau$ in $\Aut(\oQ/K)$}\}.
\end{equation}
The image of $\spann_{\Q} \G_K$ in $\F$ is the $\Q$-linear subspace $\D_K$.  Using (\ref{form345}), the image in $\F$ of the 
set on the right of (\ref{form363}) is the set on the right of (\ref{form355}).  Hence (\ref{form355}) follows from (\ref{form363}).

As each linear isometry $\Phi_{\tau}$ is continuous, the subset on the right of (\ref{form360}) is closed, and by
what we have just proved it contains $\D_K$.  Therefore we have
\begin{equation}\label{form410}
\closure \D_K \subseteq \big\{F \in \X : \text{$\Phi_{\tau}(F) = F$ for each $\tau$ in $\Aut(\oQ/K)$}\big\}.
\end{equation}

Assume that $F$ belongs to $\X$, and $F$ satisfies
\begin{equation}\label{form415}
 \Phi_{\tau}(F) = F\quad\text{for each $\tau$ in $\Aut(\oQ/K)$}.
\end{equation}
As $\F$ is dense in $\X$, for each $\ep > 0$ there exists $f_{\beta}$ in $\F$ such that
\begin{equation*}\label{form420}
\|F - f_{\beta}\|_1 < \ep.
\end{equation*}
Using (\ref{form415}) and the fact that $\Phi_{\tau}$ is an isometry, we find that
\begin{equation}\label{form425}
\big\|F - \Phi_{\tau}\bigl(f_{\beta}\bigr)\big\|_1 
	= \big\|\Phi_{\tau}\bigl(F - f_{\beta}\bigr)\big\|_1 = \big\|F - f_{\beta}\big\|_1 < \ep
\end{equation}
for each automorphism $\tau$ in $\Aut(\oQ/K)$.  Alternatively, (\ref{form345}) and (\ref{form425}) imply that
\begin{equation}\label{form430}
\|F - f_{\tau \beta}\|_1 < \ep,\quad\text{for each $\tau$ in $\Aut(\oQ/K)$}.
\end{equation}
Let $\delta_K(\beta) = L$, and let
\begin{equation*}\label{form434}
\Orb_K(\beta) = \big\{\beta_1, \beta_2, \dots , \beta_L\big\}
\end{equation*}
be the distinct elements in the orbit of $\beta$, as in (\ref{new125}).  Write
\begin{equation*}\label{from438}
\gamma = \beta_1 \beta_2 \cdots \beta_L,
\end{equation*}
so that $\gamma$ belongs to $\spann_{\Q} \G_K$.  Then
\begin{equation*}\label{form442}
L^{-1} f_{\gamma}(y) = L^{-1} \sum_{\ell = 1}^L f_{\beta_{\ell}}(y)
\end{equation*}
belongs to $\D_K$.  Using (\ref{form430}) we get
\begin{equation}\label{form450}
\begin{split}
\|F - L^{-1} f_{\gamma}\|_1 &= \bigg\|L^{-1} \sum_{\ell = 1}^L \bigl(F - f_{\beta_{\ell}}\bigr)\bigg\|_1\\
	&\le L^{-1} \sum_{\ell = 1}^L \big\|F - f_{\beta_{\ell}}\big\|_1\\
	&< \ep.
\end{split}
\end{equation}
As $\ep > 0$ was arbitrary, the inequality (\ref{form450}) implies that $F$ is a limit point of $\D_K$, and 
therefore $F$ belongs to $\closure \D_K$.  We have shown that
\begin{equation}\label{form455}
\big\{F \in \X : \text{$\Phi_{\tau}(F) = F$ for each $\tau$ in $\Aut(\oQ/K)$}\big\} \subseteq \closure \D_K.
\end{equation}
Now (\ref{form360}) follows from (\ref{form410}) and (\ref{form455}).
\end{proof}

For each algebraic field extension $K/\Q$, we use the collection of isometries on the left of (\ref{form350}) to define 
a surjective, continuous, linear projection $U_K : \X \rightarrow \X_K$.  For finite extensions of $\Q$ this was done 
by Fili and Miner in \cite[section 2.3]{FM2013}.  Here we assume only that $K/\Q$ is an algebraic extension.  

The subgroup $\Aut(\oQ/K)$ is a compact, topological group.  Let $\nu_K$ denote a Haar measure 
on the Borel subsets of $\Aut(\oQ/K)$, normalized so that 
\begin{equation*}\label{form460}
\nu_K\bigl(\Aut(\oQ/K)\bigr) = 1.
\end{equation*}
For each $F$ in $\X$, Fubini's theorem and (\ref{form120}) imply that
\begin{align}\label{ch1}
\begin{split}
\int_Y \int_{\Aut(\oQ/K)} &\bigl|\Phi_{\tau}(F)(y)\bigr|\ \dnu_K(\tau)\ \dla(y)\\
	&= \int_{\Aut(\oQ/K)} \int_Y \bigl|\Phi_{\tau}(F)(y)\bigr|\ \dla(y)\ \dnu_K(\tau) = \|F\|_1.
\end{split}
\end{align}
It follows that for $\lambda$-almost all points $y$ in $Y$ the map
\begin{equation*}\label{ch2}
\tau \mapsto \Phi_{\tau}(F)(y)
\end{equation*}
is $\nu_K$-integrable.  If $F$ belongs to $\X$, we define $U_K(F) : Y \rightarrow \R$ at $\lambda$-almost all points $y$
in $Y$ by
\begin{equation}\label{ch3}
U_K(F)(y) = \int_{\Aut(\oQ/K)} \Phi_{\tau}(F)(y)\ \dnu_K(\tau) = \int_{\Aut(\oQ/K)} F(\tau^{-1} y)\ \dnu_K(\tau).
\end{equation}
By our previous remarks, $y \mapsto U_K(F)(y)$ {\it is} finite $\lambda$-almost everywhere on $Y$.  Using (\ref{ch1}) we find that
\begin{equation}\label{ch4}
\int_Y\ |U_K(F)(y)|\ \dla(y) \le \int_Y \int_{\Aut(\oQ/K)} \bigl|\Phi_{\tau}(F)(y)\bigr|\ \dnu_K(\tau) \dla(y) = \|F\|_1,
\end{equation}
and therefore $y \mapsto U_K(F)(y)$ determines an element of $L^1(Y, \B, \lambda)$.  Because $F$ belongs to
$\X$, (\ref{form325}), (\ref{ch3}), and Fubini's theorem, imply that
\begin{equation}\label{ch5}
\int_Y U_K(F)(y)\ \dla(y) = \int_{\Aut(\oQ/K)} \int_Y \Phi_{\tau}(F)(y)\ \dla(y)\ \dnu_K(\tau) = 0.
\end{equation}
It follows from (\ref{ch5}) that $U_K(F)(y)$ belongs to $\X$, and (\ref{ch4}) shows that
\begin{equation*}\label{ch7}
U_K : \X \rightarrow \X
\end{equation*}
is a continuous, linear operator.  Next we show that $U_K$ is a projection.

\begin{theorem}\label{thmcomp1}  Let $K/\Q$ be an algebraic extension of fields, and let $U_K : \X \rightarrow \X$ be the 
continuous, linear operator defined by {\rm (\ref{ch3})}.  Then $U_K$ is a continuous, linear projection onto the
closed, $\R$-linear subspace $\X_K = \closure \D_K$, and therefore
\begin{equation}\label{form470}
\X = \X_K \oplus \Y_K,
\end{equation}
where $\Y_K = \ker U_K$.
\end{theorem}

\begin{proof}  If $F$ belongs to $\X$ and $\sigma$ is in $\Aut(\oQ/K)$, then by the translation invariance of the Haar measure 
$\nu_K$ we have
\begin{equation}\label{form500}
\begin{split}
\Phi_{\sigma}\bigl(U_K(F)\bigr)(y) &= \int_{\Aut(\oQ/K)} F\bigl(\tau^{-1} \sigma^{-1} y\bigr)\ \dnu_K(\tau)\\
						    &= \int_{\Aut(\oQ/K)} F\bigl(\tau^{-1} y\bigr)\ \dnu_K(\tau)\\
                                                            &= U_K(F)(y).
\end{split}
\end{equation}
As (\ref{form500}) holds for all $\sigma$ in $\Aut(\oQ/K)$, we get
\begin{equation*}\label{form505}
\begin{split}
U_K\bigl(U_K(F)\bigr)(y) &= \int_{\Aut(\oQ/K)} \int_{\Aut(\oQ/K)} F\bigl(\tau^{-1} \sigma^{-1} y\bigr)\ \dnu_K(\tau)\ \dnu_K(\sigma)\\
	&= \int_{\Aut(\oQ/K)} U_K(F)(y)\ \dnu_K(\sigma)\\
	&= U_K(F)(y),
\end{split}
\end{equation*}
and this shows that $U_K$ is a projection.

Because (\ref{form500}) holds for all $\sigma$ in $\Aut(\oQ/K)$, (\ref{form360}) implies that 
\begin{equation*}\label{form510}
\big\{U_K(F) : F \in \X\big\} \subseteq \X_K.
\end{equation*}
If $F$ belongs to $\X_K$ we appeal to (\ref{form360}) again, and conclude that $\Phi_{\tau}(F) = F$ for all $\tau$ in 
$\Aut(\oQ/K)$.  It follows from (\ref{ch3}) that $U_K(F) = F$, and therefore
\begin{equation*}\label{form515}
\big\{U_K(F) : F \in \X\big\} = \X_K.
\end{equation*}
We have shown that $U_K$ is a continuous, linear projection from $\X$ onto $\X_K$.  Therefore $\X_K$ is a complemented 
subspace of $\X$, and (\ref{form470}) follows immediately.
\end{proof}

Next we write
\begin{equation*}\label{form550}
S_K : \F \rightarrow \X_K
\end{equation*}
for the restriction of $U_K$ to the $\Q$-linear subspace $\F$. 

\begin{theorem}\label{thmcomp2}   Let $K/\Q$ be an algebraic extension of fields. 
Then $S_K$ is a continuous, linear projection of $\F$ onto $\D_K$, and we have the direct sum decomposition
\begin{equation}\label{form551}
\F = \D_K \oplus \E_K,
\end{equation}
where
\begin{equation*}\label{form552}
\E_K = \ker S_K = \big\{f_{\alpha} \in \F : S_K\bigl(f_{\alpha}\bigr) = 0 \big\}.
\end{equation*}
Moreover, we have $\closure \D_K = \X_K$, $\closure \E_K = \Y_K$, and both $\D_K$ and $\E_K$ are $\F$-closed.
\end{theorem}

\begin{proof}  Since $U_K$ is a continuous, linear projection, it is trivial that the restriction $S_K$ is a continuous,
linear projection from $\F$ onto a $\Q$-linear subspace contained in $\X_K$.  Let $\beta$ be an element of $\G$
such that $\delta_K(\beta) = L$, and let
\begin{equation}\label{form554}
\Orb_K(\beta) = \big\{\beta_1, \beta_2, \dots , \beta_L\big\}
\end{equation}
be the distinct elements in the orbit of $\beta$, as in (\ref{new125}).  Write
\begin{equation*}\label{from558}
\gamma = \beta_1 \beta_2 \cdots \beta_L,
\end{equation*}
so that $\gamma$ belongs to $\spann_{\Q} \G_K$, and 
\begin{equation*}\label{form560}
L^{-1} f_{\gamma}(y) = L^{-1} \sum_{\ell = 1}^L f_{\beta_{\ell}}(y)
\end{equation*}
belongs to $\D_K$. Using (\ref{form345}) and (\ref{form554}) we find that
\begin{equation*}\label{form565}
\begin{split}
S_K\bigl(f_{\beta}\bigr)(y) &= U_K\bigl(f_{\beta}\bigr)(y)\\
					 &= \int_{\Aut(\oQ/K)} \Phi_{\tau}\bigl(f_{\beta}\bigr)(y)\ \dnu_K(\tau)\\
					 &= L^{-1} \sum_{\ell = 1}^L f_{\beta_{\ell}}(y)\\
					 &= L^{-1} f_{\gamma}(y).
\end{split}
\end{equation*}
It follows that the image of $S_K$ is contained in $\D_K$.  But $U_K$, and therefore also $S_K$, fixes 
each element of $\D_K \subseteq \X_K$.  We conclude that $S_K$ is a continuous, linear 
projection of $\F$ onto $\D_K$.  The direct sum decomposition (\ref{form551}) follows immediately.

Because $S_K : \F \rightarrow \D_K$ is a continuous, linear projection, Theorem \ref{thmsum1} implies that
\begin{equation}\label{form570}
\X = \closure \D_K \oplus \closure \E_K.
\end{equation}
As $\closure \D_K = \X_K$ is the image of $U_K$, we conclude from (\ref{form570}) that 
\begin{equation*}\label{form575}
\closure \E_K = \Y_K = \ker U_K.
\end{equation*}
Then it follows from Theorem \ref{thmsum2} that  both $\D_K$ and $\E_K$ are $\F$-closed.
\end{proof}

Let $\beta$ be an element of $\G$ with $\delta_K(\beta) = L$.  Using (\ref{new125}) we have
\begin{equation*}\label{form580}
\Orb_K(\beta) = \big\{\beta_1, \beta_2, \dots , \beta_L\big\}.
\end{equation*}
As in the proof of Theorem \ref{thmcomp2} we find that
\begin{equation*}\label{form585}
S_K\bigl(f_{\beta}\bigr)(y) = L^{-1} f_{\gamma}(y),
\end{equation*}
where
\begin{equation*}\label{from590}
\gamma = \beta_1 \beta_2 \cdots \beta_L,
\end{equation*}
and $\gamma$ belongs to $\spann_{\Q} \G_K$.  We conclude that $f_{\beta}(y)$ belongs to the $\Q$-linear 
subspace $\E_K$ if and only if $\gamma = 1$.  That is, $f_{\beta}(y)$ belongs to $\E_K$ if and only if
\begin{equation*}\label{form595}
f_{\beta_1}(y) + f_{\beta_2}(y) + \cdots + f_{\beta_L}(y) = 0
\end{equation*}
at each point $y$ in $Y$.
\section{A sufficient condition for projections to commute}\label{ptc}

Let $K/\Q$ and $L/\Q$ be algebraic field extensions contained in a common algebraic closure $\oQ$.  Let
\begin{equation*}\label{chp5}
U_K : \X \rightarrow \X_K,\quad\text{and}\quad U_L : \X \rightarrow \X_L,
\end{equation*}
be the corresponding surjective, continuous, linear projections defined by (\ref{ch3}).  In this section we identify a condition on 
the fields $K$ and $L$ which implies that the projections $U_K$ and $U_L$ commute.  

\begin{lemma}\label{lemchp1}  Let $K/\Q$ and $L/\Q$ be algebraic field extensions, and let $U_K$, and $U_L$, be the 
corresponding continuous, linear projections defined by {\rm (\ref{ch3})}.  Assume that $\sigma$ is an automorphism in 
$\Aut(\oQ/\Q)$ such that $\sigma K = L$.  Then we have
\begin{equation}\label{chp45}
\Phi_{\sigma}\circ U_K = U_L\circ \Phi_{\sigma}.
\end{equation}  
\end{lemma}

\begin{proof}  The operators $U_K : \X \rightarrow \X_K$, $U_L : \X \rightarrow \X_L$, and 
$\Phi_{\sigma} : \X \rightarrow \X$, are all continuous, and $\F$ is dense in $\X$.  Hence it suffices to prove the identity
\begin{equation*}\label{chp47}
\Phi_{\sigma}\bigl(U_K\bigl(f_{\alpha}\bigr)\bigr)(y) = U_L\bigl(\Phi_{\sigma}\bigl(f_{\alpha}\bigr)\bigr)(y) 
\end{equation*}
for each function $f_{\alpha}$ in $\F$, and each point $y$ in $Y$.   Because $\sigma K = L$, the subgroups $\Aut(\oQ/K)$ and
$\Aut(\oQ/L)$ are conjugate subgroups of $\Aut(\Q/\Q)$.  More precisely, we find that
\begin{equation*}\label{chp49}
\sigma^{-1} \Aut(\oQ/L) \sigma = \Aut(\oQ/K).
\end{equation*}
Moreover, for each Borel subset $E \subseteq \Aut(\oQ/L)$ the normalized Haar measures $\nu_K$ and $\nu_L$ are related 
by the identity
\begin{equation*}\label{chp50}
\nu_K\bigl(\sigma^{-1} E \sigma\bigr) = \nu_L(E).
\end{equation*}
Then for each function $f_{\alpha}$ in $\F$, and each point $y$ in $Y$, we find that
\begin{align}\label{chp51}
\begin{split}
\Phi_{\sigma^{-1}}\bigl(U_L\bigl(\Phi_{\sigma}\bigl(f_{\alpha}\bigr)\bigr)\bigr)(y) 
	&= \int_{\Aut(\oQ/L)} f_{\alpha}\bigl(\sigma^{-1}\tau^{-1}\sigma y\bigr)\ \dnu_L(\tau)\\
	&= \int_{\sigma^{-1} \Aut(\oQ/L) \sigma}f_{\alpha}\bigl(\tau^{-1} y\bigr)\ \dnu_K(\tau)\\
	&= U_K\bigl(f_{\alpha}\bigr)(y).
\end{split}
\end{align}
By applying the operator $\Phi_{\sigma}$ to both sides of (\ref{chp51}), we get (\ref{chp45}).
\end{proof}

\begin{theorem}\label{thmchp1}  Let $K/\Q$ and $L/\Q$ be algebraic field extensions, and let $U_K$, and $U_L$, be the 
corresponding continuous, linear projections defined by {\rm (\ref{ch3})}.  If either $K/(K\cap L)$ is a (possibly infinite) 
Galois extension, or if $L/(K \cap L)$ is a (possibly infinite) Galois extension, then we have
\begin{equation}\label{chp65}
U_K\circ U_L = U_L\circ U_K.
\end{equation}  
\end{theorem}

\begin{proof}  We assume that $K/(K \cap L)$ is a Galois extension.  As in the proof of Lemma \ref{lemchp1}, it suffices to 
verify the identity
\begin{equation}\label{chp67}
U_K\bigl(U_L\bigl(f_{\alpha}\bigr)\bigr)(y) = U_L\bigl(U_K\bigl(f_{\alpha}\bigr)\bigr)(y)
\end{equation}
for each function $f_{\alpha}$ in $\F$, and each point $y$ in $Y$.  Let $\sigma$ be an automorphism in $\Aut(\oQ/K\cap L)$,
so that $\sigma K = K$.  Then by Lemma \ref{lemchp1} we have  
\begin{equation}\label{chp69}
\Phi_{\sigma}\bigl(U_K\bigl(f_{\alpha}\bigr)\bigr)(y) = U_K\bigl(\Phi_{\sigma}\bigl(f_{\alpha}\bigr)\bigr)(y) 
\end{equation}
for each function $f_{\alpha}$ in $\F$, and each point $y$ in $Y$.  Using (\ref{ch3}), the identity (\ref{chp69}) can be written as
\begin{equation}\label{chp71}
\int_{\Aut(\oQ/K)} f_{\alpha}\bigl(\tau^{-1} \sigma^{-1} y \bigr)\ \dnu_K(\tau)
	= \int_{\Aut(\oQ/K)} f_{\alpha}\bigl(\sigma^{-1} \tau^{-1} y \bigr)\ \dnu_K(\tau).
\end{equation}
As $\Aut(\oQ/L) \subseteq \Aut(\oQ/K\cap L)$, we can integrate both sides of (\ref{chp71}) over automorphisms $\sigma$ in 
the subgroup $\Aut(\oQ/L)$.  Applying Fubini's theorem, we arrive at the identity
\begin{align}\label{chp75}
\begin{split}
\int_{\Aut(\oQ/L)}\ \bigg\{&\int_{\Aut(\oQ/K)} f_{\alpha}\bigl(\tau^{-1} \sigma^{-1} y \bigr)\ \dnu_K(\tau)\bigg\}\ \dnu_L(\sigma)\\
	&= \int_{\Aut(\oQ/L)}\ \bigg\{\int_{\Aut(\oQ/K)} f_{\alpha}\bigl(\sigma^{-1} \tau^{-1} y \bigr)\ \dnu_K(\tau)\bigg\}\ \dnu_L(\sigma)\\
	&= \int_{\Aut(\oQ/K)}\ \bigg\{\int_{\Aut(\oQ/L)} f_{\alpha}\bigl(\sigma^{-1} \tau^{-1} y \bigr)\ \dnu_L(\sigma)\bigg\}\ \dnu_K(\tau).
\end{split}
\end{align}
Equality between the first and third iterated integrals in (\ref{chp75}) is exactly the identity (\ref{chp67}).  This verifies
(\ref{chp65}).
\end{proof}

\section{Statement and proof of the main theorem}\label{finalproof}

For each algebraic extension $K/\Q$ we continue to write
\begin{equation*}\label{part0}
S_K : \F \rightarrow \D_K
\end{equation*}
for the restriction of $U_K$ to the $\Q$-linear subspace $\F$.  We prove the following general result, which includes 
Theorem \ref{thmintro3} and Theorem \ref{thmintro4} as special cases.

\begin{theorem}\label{thmpart1}  Let $K_1, K_2, \dots , K_N$, be a collection of fields such that
\begin{equation*}\label{part5}
\Q \subseteq K_n \subseteq \oQ,\quad\text{for each $n = 1, 2, \dots , N$},
\end{equation*}
and for each pair of integers $n_1 \not= n_2$, either
\begin{equation}\label{part10}
K_{n_1}/(K_{n_1} \cap K_{n_2}) \quad\text{is a (possibly infinite) Galois extension,}
\end{equation}
or
\begin{equation}\label{part15}
K_{n_2}/(K_{n_1} \cap K_{n_2}) \quad\text{is a (possibly infinite) Galois extension.}
\end{equation}
For each $n = 1, 2, \dots , N$, let
\begin{equation*}\label{part20}
\F = \D_{K_n} \oplus \E_{K_n}
\end{equation*}
be the direct sum decomposition determined by the surjective, continuous, linear projection
\begin{equation*}\label{part25}
S_{K_n} : \F \rightarrow \D_{K_n},\quad\text{with}\quad \E_{K_n} = \ker S_{K_n}.
\end{equation*}
Then for $0 \le M \le N$, the $\Q$-linear subspace
\begin{equation*}\label{part30}
\D_{K_1} + \D_{K_2} + \cdots + \D_{K_M} + \E_{K_{M+1}} + \E_{K_{M+2}} + \cdots + \E_{K_N}
\end{equation*}
is $\F$-closed. 
\end{theorem}

\begin{proof}  For each $n = 1, 2, \dots , N$, we define
\begin{equation*}\label{part65}
T_{K_n} : \F \rightarrow \F
\end{equation*}
by
\begin{equation*}\label{part70}
T_{K_n}\bigl(f_{\alpha}\bigr)(y) = f_{\alpha}(y) - S_{K_n}\bigl(f_{\alpha}\bigr)(y).
\end{equation*}
Then each map
\begin{equation*}\label{part75}
T_{K_n} : \F \rightarrow \E_{K_n}
\end{equation*}
is a surjective, continuous, linear projection, such that
\begin{equation*}\label{part80}
\D_{K_n} = \ker T_{K_n}.
\end{equation*}
The hypotheses (\ref{part10}) and (\ref{part15}), together with Theorem \ref{thmchp1}, imply that for each $m \not= n$ the 
projections $U_{K_m}$ and $U_{K_n}$ commute.  Hence the restricted projections $S_{K_m}$ and $S_{K_n}$ also 
commute.  One easily checks that $S_{K_m}$ and $T_{K_n}$ commute, and similarly that $T_{K_m}$ and $T_{K_n}$
commute.  Thus we have the collection of continuous projects
\begin{equation}\label{part85}
\big\{S_{K_1}, S_{K_2}, \dots , S_{K_M}, T_{K_{M+1}}, T_{K_{M+2}}, \dots , T_{K_N}\big\},
\end{equation}
and each pair of projections from (\ref{part85}) commute.  These correspond to the collection of direct sum decompositions
\begin{equation*}\label{part90}
\F = \D_{K_m} \oplus \E_{K_m},\quad\text{for $m = 1, 2, \dots , M$},
\end{equation*}
and
\begin{equation*}\label{part93}
\F = \E_{K_n} \oplus \D_{K_n},\quad\text{for $n = M+1, M+2, \dots , N$}.
\end{equation*}
We have verified all the hypotheses of Theorem \ref{thmform1} with
\begin{equation*}\label{part95}
\hH_m = \D_{K_m},\quad\text{and}\quad \I_m = \E_{K_m},\quad\text{for $m = 1, 2, \dots , M$},
\end{equation*} 
and with
\begin{equation*}\label{part100}
\hH_n = \E_{K_n},\quad\text{and}\quad \I_n = \D_{K_n},\quad\text{for $n = M+1, M+2, \dots , N$}.
\end{equation*} 
From the conclusion of Theorem \ref{thmform1} we get the direct sum decomposition
\begin{equation*}\label{part105}
\begin{split}
\F =  \bigl(\D_{K_1} + \D_{K_2} &+ \cdots + \D_{K_M} + \E_{K_{M+1}} + \E_{K_{M+2}} + \cdots + \E_{K_N}\bigr) \\
	& \oplus \bigl(\E_{K_1} \cap \E_{K_2} \cap \cdots \cap \E_{K_M} \cap \D_{K_{M+1}} \cap \D_{K_{M+2}} \cap \D_{K_N}\bigr).
\end{split}
\end{equation*}
We also learn from Theorem \ref{thmform1} that the unique, surjective, linear projection 
\begin{equation*}\label{prf67}
W_N : \F \rightarrow \D_{K_1} + \D_{K_2} + \cdots + \D_{K_M} + \E_{K_{M+1}} + \E_{K_{M+2}} + \cdots + \E_{K_N},
\end{equation*}
such that 
\begin{equation*}\label{prf69}
\ker W_N = \E_{K_1} \cap \E_{K_2} \cap \cdots \cap \E_{K_M} \cap \D_{K_{M+1}} \cap \D_{K_{M+2}} \cap \D_{K_N},
\end{equation*}
is continuous.  Because $W_N$ is continuous, it follows from Theorem \ref{thmsum2} that
\begin{equation*}\label{prf77}
\D_{K_1} + \D_{K_2} + \cdots + \D_{K_M} + \E_{K_{M+1}} + \E_{K_{M+2}} + \cdots + \E_{K_N}
\end{equation*}
is $\F$-closed.  This completes the proof of Theorem \ref{thmpart1}.
\end{proof}



\begin{thebibliography}{1}

\bibitem{all2009}
D.~Allcock and J.~D.~Vaaler,
\newblock A Banach space determined by the Weil height,
\newblock {\em Acta Arith.}, 136 (2009), 279--298.

\bibitem{bombieri2006}
E.~Bombieri and W.~Gubler,
\newblock {\em Heights in Diophantine Geometry},
\newblock Cambridge U. Press, New York, 2006.

\bibitem{dlmf2008}
A.~C.~de~la Maza and E.~Friedman.
\newblock Heights of algebraic numbers modulo multiplicative group actions.
\newblock {\em J. Number Theory}, 128 (2008), 2199--2213.

\bibitem{FM2012}
P.~Fili and Z.~Miner,
\newblock Norms extremal with respect to the Mahler measure,
\newblock {\em J. Number Theory} 132 (2012), 275--300.

\bibitem{FM2013}
P.~Fili and Z.~Miner,
\newblock Orthogonal decomposition of the space of algebraic numbers and Lehmer's problem,
\newblock {\em J. Number Theory} 133 (2013), 3941--3981.

\bibitem{lamperti1958}
J.~Lamperti
\newblock On the isometries of certain function-spaces,
\newblock {\em Pacific J. Math.} 8, (1958), 459--466.

\bibitem{megginson}
R.~E.~Megginson,
\newblock {\em An Introduction to Banach Space Theory},
\newblock Springer-Verlag, New York, 1998.

\bibitem{vaaler2012}
J.~D.~Vaaler,
\newblock Heights on groups and small multiplicative dependencies,
\newblock {\em Trans. Amer. Math. Soc.} 366 (2014), no. 6, pp. 3295--3323.

\end{thebibliography}
\end{document}